%% file: arxiv.tex
\documentclass[11pt,reqno]{amsart}
\usepackage[utf8]{inputenc}
\usepackage[english,russian]{babel}

\usepackage{caption}
\usepackage{amsmath}
\usepackage{amssymb}
\usepackage{amsfonts}
\usepackage{graphicx}
\usepackage{xcolor}
\usepackage[colorlinks]{hyperref}
\usepackage{tikz}
\usetikzlibrary{decorations.pathreplacing,calligraphy}
\newlength{\dk} \dk=1.7em
\newlength{\dktmp} \dktmp=0.7\dk
\tikzset{x=1.0\dk,y=1.0\dk,
 cell/.style={circle,draw=black, thin, minimum size=0.77\dk},
 }

\def\cl#1#2{ ++(1.2,0) node [cell,fill=#1] {\raisebox{-0.25em}[0.2em][0em]{\makebox[0mm][c]{#2}}} }

\usepackage{titlesec}
\titlespacing{\section}{0pt}{2.5ex}{1.5ex}
\titlespacing{\subsection}{0pt}{1.5ex}{1ex}
\titlespacing{\subsubsection}{0pt}{1.5ex}{1ex}
\titleformat{\section}{\large\bfseries\centering}{\thesection}{1em}{}
\titleformat{\subsection}[runin]{\bfseries}{\thesubsection.}{0.5em}{}[.\mbox{\ }]
\titleformat{\subsubsection}[runin]{\bfseries}{\thesubsubsection.}{0.4em}{}[.\mbox{\ }]

\newtheorem{construction}{Construction}
\newcommand\orcidicon[1]{\href{https://orcid.org/#1}{\includegraphics[scale=0.02]{orcid.pdf}}}

\newtheorem{lemma}{Lemma}

\newtheorem{remark}{Remark}
\newtheorem{theorem}{Theorem}

\newtheorem{corollary}{Corollary}
\newtheorem{proposition}{Proposition}

\hypersetup{linkcolor=blue,filecolor=black,urlcolor=blue, citecolor=blue}

\begin{document}
\renewcommand{\refname}{References}
\renewcommand{\proofname}{Proof.}
\renewcommand{\figurename}{Fig.}

\thispagestyle{empty}

\title[On null CRCs in grid]{\large On null completely regular codes in Manhattan metric}

\author[I.YU. MOGILNYKH]{{\bf I.Yu. Mogilnykh}}
\author[A.YU. VASIL'EVA]{{\bf A.Yu. Vasil'eva}}

\address{Ivan Yurevich Mogilnykh  
\newline\hphantom{iii} Sobolev Institute of Mathematics,
\newline\hphantom{iii} pr. Koptyuga, 4,
\newline\hphantom{iii} 630090, Novosibirsk, Russia}%
\email{\textcolor{blue}{ivmog84@gmail.com}}%

\address{Anastasia Yurievna Vasil'eva  
\newline\hphantom{iii} Sobolev Institute of Mathematics,
\newline\hphantom{iii} pr. Koptyuga, 4,
\newline\hphantom{iii} 630090, Novosibirsk, Russia}%
\email{\textcolor{blue}{vasilan@math.nsc.ru}}%

\thanks{\rm The work of the both authors was performed according to the Government
research assignment for IM SB RAS, Project No. FWNF-2022-0017.}

\vspace{1cm}
\maketitle {\small

\vspace{-12pt}

\bigskip
\bigskip

\begin{quote}
\noindent{\bf Abstract: We investigate the class of completely regular codes in graphs with a distance partition $C_0, \ldots, C_{\rho}$, where each set $C_i$, for $0 \leq i \leq r - 1$, is an independent set. This work focuses on the existence problem for such codes in the $n$-dimensional infinite grid. We demonstrate that several parameter families of such codes necessarily arise from binary or ternary Hamming graphs or  do not exist. Furthermore, employing binary linear programming tech\-niques, we explore completely regular codes in infinite grids of dimensions $3$ and $4$ for the cases $r = 1$ and $r = 2$.}    
 \medskip

 \end{quote}
}

\bigskip

\section{Introduction}

One-error-correcting perfect codes \cite{GW70} and diameter perfect codes \cite{Et11} in the $n$-dimensional infinite  grid graph are well-studied objects; both can be defined as completely regular codes with covering radii $1$ and $2$.
Later variations and generalizations can be viewed somewhat separately and include perfect colorings with multiple colors and completely regular codes (the latter is referred to as CRC) with various covering radii. 
Perfect $k$-colorings in the infinite  grid graph $G_n$ were studied by Axenovich for $k = n = 2$ \cite{Ax03}, by Puzynina for $k = 3$, $n = 2$ \cite{Puz05} and by Dobrec et al.\ for $k = 2$ and arbitrary $n \geq 2$ \cite{DGHM09}. For $k \leq 9$ and $n = 2$, they were studied by Krotov. Completely regular codes in $G_n$ were considered in \cite{Tarasov} and \cite{AvgVas22}.

As the infinite grid graph provides a graph covering for the binary even-length Hamming and modular grid graphs, the characterization problem for CRCs with putative parameters in the infinite rectangular grid graph $G_n$ is at least as hard as the respective long-standing open problems \cite{BRZ19}, \cite{KKM25} for CRCs in classical binary and ternary Hamming graphs.

We study completely regular codes whose first $r$ diagonal elements of parameter matrix are zeros; we refer to such codes as $r$-null CRCs. Several important classes of error-correcting codes in Hamming and grid graphs such as perfect codes, diameter perfect codes, Golay and Preparata codes and their derived codes are $r$-null completely regular with maximal or premaximal possible $r$ \cite{BRZ19}. In a similar fashion, one can define all-ones CRCs.

The next three sections are introductory: basic definitions are provided in Section 2; constructions of completely regular codes in the $n$-dimensional grid graph are reviewed in Section 3 and the structure of the first four layers of the graph is discussed in Section 4.
In Section 5, for arbitrary $n$ and any covering radius, we show that any $2$-null CRC in $G_n$ with $c_1 = 1$ and $c_2 = 2$ necessarily arises from the binary Hamming graph and that $2$-null CRCs with $c_1 = 1$ or $2$ and $c_3 = 3$ exist only in the case $c_1 = 1$, $c_2 = 3$ and $n = 2$.
In Section 6, we prove that $1$-null completely regular codes with $c_2 = 1$ necessarily have period three.
Codes with the all-ones diagonal in the parameter matrix are considered in Section 7, where several subclasses of these codes are also shown either to be derived from the binary Hamming graph or not to exist.

In Section 8, we propose a binary linear programming approach to the existence problem for completely regular codes in the $n$-dimensional infinite grid for $n = 3$, $4$. We characterize all $1$-null CRCs in $G_3$ up to their parameters and rule out several $2$-null CRCs in $G_4$. A distinctive feature of our approach is a relaxed form of the linear programming problem: the infeasibility of a single instance of linear programming  problem implies the nonexistence of multiple codes with different covering radii.

\section{Definitions and notations}
A vertex partition $(C_0, C_1, \ldots, C_{k-1})$ of a graph is called a perfect $k$-coloring (or equitable $k$-partition) if for every $i, j \in {0, 1, \ldots, k-1}$ there exists an integer $\alpha_{ij}$ such that every vertex in $C_i$ has exactly $\alpha_{ij}$ neighbors in $C_j$.
The matrix $A = (\alpha_{ij})$ is called the parameter matrix of the perfect coloring.
In the case when the parameter matrix is tridiagonal, the set $C_0$ is called a completely regular code.
The maximum distance between the code and the remaining vertices is called the covering radius of the code and is denoted by $\rho$.
We will use both the classical representation of the parameter matrix and the following compact notation as the sequence of its tridiagonal elements:
$$[a_0,b_0 | c_1,a_1,b_1 | c_2,a_2,b_2 | \ldots | c_{\rho},a_{\rho}].$$  
Given a code, the set of vertices at distance equal to the covering radius from the code is called the opposite code.
Obviously, if a code is a CRC, then its opposite code is also a CRC and vice versa.

A CRC with parameter matrix $A$ is called an $r$-null CRC if $a_0 = \ldots = a_{r-1} = 0$.
A CRC is called an all-nulls (respectively, all-ones) CRC if $a_0 = \ldots = a_{\rho} = 0$ (respectively, $a_0 = \ldots = a_{\rho} = 1$).

Note that, by definition, $1$-null CRCs are CRCs with minimum distance at least two.
Moreover, $2$-null CRCs in Hamming graphs include all CRCs with minimum distance at least four and, in particular, many well-known families of optimal codes, such as Preparata, BCH, Golay codes, perfect codes and shortened perfect codes.
All these codes give rise to CRCs in the $n$-dimensional grid (see Section 3 below).

The $n$-dimensional infinite rectangular grid is the graph $G_n$ with vertex set ${\bf Z}^n$, where ${\bf Z}$ denotes the set of integers and edge set $$\{ (x,y) : \sum^n_{i=1}|x_i-y_i| = 1\}.$$
The graph metric in  $G_n$ is often called Manhattan metric \cite{Et11}. 

For simplicity, we use non-comma notation for elements in ${\bf Z}^n$, for example $(101\text{-}21)$ designates an element of ${\bf Z}^5$.  

By $G_{n,q}$ we denote the graph with the vertex set $\{0,\ldots,q-1\}^n$ and 
the edge set $$\{ (x,y) : \sum^n_{i=1}min\{|x_i-y_i|,|q-x_i+y_i| \}= 1\}.$$  The vertices of the Hamming graph $H(n,q)$ are $\{0,\ldots,q-1\}^n$, the edge set is $$\{ (x,y) : |\{i\in \{1,\ldots,n\}:x_i\neq y_i\}| = 1\}.$$ We use a word to  refer to a vertex of these graphs.  
The all-zero word in denoted by ${\bf 0}$. 
The weight of the word $x$ is its  distance between $x$ and ${\bf 0}$ in $G_n$.  
Throughout the paper we denote by $e_i$ the word of weight one with $i$th  symbol equal to $1$. The vertex set of infinite triangular grid is  ${\bf Z}^2$ and the edge set is 
$$\{ (x,y) : \sum^2_{i=1}|x_i-y_i| = 1\}\cup \{(x,y):x_1-y_1=x_2-y_2\in \{-1,1\}\}.$$
Both $G_n$ and the infinite triangular grid are Cayley graphs: $G_n$ is the Cayley graph of the $n$th power of the additive group of integers and the infinite triangular grid is the Cayley graph of its second power.
The respective Cayley graph generator sets are 
$\{\pm e_i: i \in \{1,\ldots,n\}\}$ and $\{\pm (10), \pm (01), \pm (11)\}$.

 The following restriction on the parameters of completely regular codes in $G_n$  was established in \cite{AvgVas22}. 
\begin{theorem}\cite[Theorem 1]{AvgVas22}\label{T_AV_1} The following holds for a CRC in $G_n$ with parameter matrix $[a_0,b_0 | c_1,a_1,b_1 | c_2,a_2,b_2 | \ldots | c_{\rho},a_{\rho}]$ and covering radius $\rho$:
$$c_1\leq c_2\leq \ldots \leq c_{\rho},$$ 
$$b_0\geq b_1\geq \ldots \geq b_{\rho-1}.$$
\end{theorem}

\section{Constructions of CRCs in $G_n$}

Below, we provide constructions of completely regular codes in the $n$-dimensional grid for the current study. The concepts behind these methods are not new. Many of them were presented earlier for CRCs in distance-regular or Cayley graphs and are well-known or can be easily generalized to the $n$-dimensional grid. Several recent surveys on CRCs \cite{BRZ19}, \cite{BKMTV21}, \cite{KKM25} provide a solid background for the constructions described below.

\subsection{ General constructions}Given two graphs $\Gamma$ and $\Gamma'$ with vertex sets $V$ and $V'$, respectively, a surjection $\phi: V \rightarrow V'$ is called a covering (and $\Gamma$ is called a covering graph) if $\phi$ bijectively maps the neighborhood of any vertex in $\Gamma$ to the neighborhood of its image. The role of coverings in the theory of perfect colorings is reflected by the following fact. 
\begin{proposition}\label{T_cover}
 Let $\phi$ be a covering  with  a graph $\Gamma$. If $(C_0,\ldots, C_{k-1})$ is a perfect coloring of $\Gamma'$ with parameter matrix $A$, then 
$(\phi^{-1}(C_0),$ $\ldots,$ $\phi^{-1}(C_{k-1}))$ is a perfect coloring of $\Gamma$ with parameter matrix $A$.
\end{proposition} 

Let $\phi:G\rightarrow G'$ be a homomorphism from the group $G$ to $G'$, which maps a generator set $S$ of $G$ to $S'$ of $G'$. Then $\phi$ is a covering of the Cayley graph of $G'$ with the generator set $S'$  with a covering graph being the Cayley graph of $G$ with the generator set $S$  \cite[Lemma 4.1]{BKMTV21}.

The component-wise modulo $q$ mapping is a homomorphism that implies that the graph $G_n$ is a covering graph for $G_{n,q}$, $q\geq 3$ and the Hamming graphs $H(2n,2)$, $H(n,3)$ in particular. Moreover, $G_3$ covers an infinite trian\-gular grid. Due to Proposition \ref{T_cover}, the CRCs in these graphs are sources for CRCs in the $n$-dimensional  grid. 

We say that a code $C$ has a period $q$ in $i$th direction if $x\pm qe_i\in C$ for $x\in C$. 

\begin{construction}\label{Constr_Gnq} Let 
$C$ be a CRC in the  graph $G_{n,q}$, $q\geq 3$ then the code $\{x\in {\bf Z}^n: (x_1 \mbox{ mod } q,\ldots, x_n \mbox{ mod } q) \in C\}$ is a CRC in $G_n$ with the same parameter matrix as $C$. 
\end{construction} 
\begin{remark}
One can show that a CRC $C$ in $G_n$ is obtained from Construction \ref{Constr_Gnq} by showing that $q$ is a period of $C$ in any direction i.e. 
$x\pm qe_i\in C$,  for all $i \in \{1,\ldots,n\}$ and $x$ in $C$. 
For any fixed $n\geq 2$, there are many  perfect (diameter perfect) codes in $G_n$ with  different periods  \cite{Et11}. In Sections 5-7  we show that some classes of null-CRCs have only  specific period $q$, e.g. $q=3$ or $4$ in each direction.

\end{remark}
We specially distinguish  two  cases of the above construction.

\begin{construction}\label{Constr_Hn3} Let 
$C$ be a CRC in the  Hamming graph $H(n,3)=G_{n,3}$, then the code $\{x\in {\bf Z}^n: (x_1 \mbox{ mod } 3,\ldots, x_n \mbox{ mod } 3) \in C\}$ is completely regular in $G_n$ with the same parameter matrix as $C$.  
\end{construction}

As the graph  $G_{n,4}$ is isomorphic to  $H(2n,2)$ via the Gray mapping $\tau(0)=(00)$, $\tau(1)=(10)$, $\tau(2)=(11)$, $\tau(3)=(01)$ we obtain the following. 

\begin{construction}\label{Constr_H2n2} If $C$ is a CRC in the Hamming graph $H(2n,2)$.  then the code $\{x\in {\bf Z}^n: (\tau(x_1 \mbox{ mod } 4),\ldots, \tau(x_n \mbox{ mod } 4)) \in C\}$ is a CRC in $G_n$ with the same parameter matrix as $C$.
\end{construction} 
Consider the following mapping: $$\phi(x_1,x_2,x_3)=(x_1-x_2,x_3-x_2)$$ for $(x_1,x_2,x_3)\in {\bf Z}^3$. This function is a group homomorphism from the additive group of ${\bf Z}^3$ to that of ${\bf Z}^2$. Moreover, it maps the Cayley generators of $G_3$ to those of the infinite triangular grid:
$\phi(e_1)=(1,0)$, $\phi(e_3)= (0,1)$, $\phi(e_2)=-(1,1)$.
Therefore, $\phi$ is a covering of the infinite triangular grid with covering graph $G_3$ and from Proposition 
\ref{T_cover} we obtain the following.

\begin{construction}\label{Constr_Triangular} If $C$ is a  CRC  in the infinite triangular grid  then 
$$\{(x_1,x_2,x_3)\in {\bf Z}^3: (x_1-x_2,x_3-x_2) \in C\}$$ is a CRC in $G_3$ with the same parameter matrix as $C$. 
\end{construction}





The construction  below was firstly  described by Fon-der-Flaas in \cite{FDF07} for Hamming graphs. It  can  be  extended to  Cayley graphs of abelian groups such as infinite grid.

\begin{construction}\label{Constr_ZMult} Let $C$ be a CRC in the graph $G_n$ (or $H(n,q)$). For integer $k\geq 1$, $x\in {\bf Z}^{kn}$  we define  $\phi(x)=(\sum\limits_{i=1,\ldots,k} x_i, \ldots,\sum\limits_{i=1,\ldots,k} x_{(n-1)k+i}).$ The code $\{x: \phi(x)\in C\}$  is a CRC in $G_{kn}$ (or $H(kn,q)$) with parameter matrix $kA$.


\end{construction}

We finish this subsection with the following important result.

\begin{theorem}\label{T_AV_2}\cite[Theorem 2]{AvgVas22} Let $D$ be a CRC with parameter matrix $A$ in $G_n$ such that $c_i=c_j$, $a_i=a_j$ and
$b_i=b_j$ for some $1\leq i<j\leq \rho-1$. Then $c_l=c_m$, $a_l=a_m$ and
$b_l=a_m$ for any $1\leq l\neq m \leq \rho-1$ and CRC $D$ is obtained by Construction \ref{Constr_ZMult} from a CRC in $G_1$. 
\end{theorem}

\subsection{Null-CRC codes}

The important classes of optimal error-correcting codes in $G_n$ include perfect and diameter perfect codes. A perfect code (also called a $1$-perfect code) in a $k$-regular graph is a set of vertices that intersects every ball of radius one in the graph in exactly one vertex. Equivalently, a perfect code is a CRC with parameter matrix $\left(%
\begin{array}{cc} 0& k \\
1& k-1 \\
\end{array}%
\right)$. Such codes were shown to exist in $G_n$ for any $n  \geq  1$ in the work of Golomb and Welsh \cite{GW70}.

\begin{construction}\label{Constr_halved_pc}  The collection of even weight words of a perfect code in $G_n$  is a completely regular code with parameter matrix $$\left(%
\begin{array}{cccc} 0& 2n & 0 & 0 \\
1& 0 & 2n-1 & 0\\
0& 2n-1 & 0 & 1 \\
0& 0 & 2n & 0 
\\\end{array}%
\right).$$   
\end{construction}

A code $C$ is called a diameter perfect code in $G_n$ if the union of any two balls of radius $1$ with adjacent centers contains exactly one vertex of $C$ \cite{Et11}. In \cite[Section II E]{Et11} the existence of such subgroup codes in $G_n$ for any $n\geq 1$ was settled. In particular, because the codes are subgroups of ${\bf Z}^n$, there is a partition of $G_n$ into disjoint diameter perfect codes.
 The union of $t$ disjoint such codes with words of the same overall parity check forms a completely regular, see \cite[Proposition 7.23]{KKM25}. 
 Such completely regular codes exist for all  $t$, $n\geq 1$.

\begin{construction}\label{Constr_diam}
 The union of $t$ disjoint diameter perfect codes in $G_n$ with words of even weight is completely regular with  parameter matrix
$$\left(%
\begin{array}{ccc} 0& 2n & 0  \\
t& 0 & 2n-t \\
 0 & 2n & 0 
\\\end{array}%
\right).$$ 

\end{construction}

As the Hamming graph is distance-regular, its distance partition  with respect to a singleton vertex is a perfect coloring. From Construction \ref{Constr_H2n2} we obtain the following.

\begin{construction}\label{Constr_H2n2distance} 
The code $\{x:x_i \mbox{ mod } 4=0, i\in \{1,\ldots,n\} \}$ is an all-nulls CRC in $G_n$ with covering radius $\rho=2n$  and  parameter matrix 

$$\left(%
\begin{array}{cccccc} 0& 2n & 0 & 0 & \ldots & 0 \\
1& 0 & 2n-1 & 0 & \ldots & 0 \\
0& 2 & 0 & 2n-2 & \ldots& 0  
\\
\ldots

\\
0 &0&\ldots & 2n-1 &0 & 1 \\

0& 0 &\ldots & 0 &2n & 0 
\\\end{array}%
\right)$$
\end{construction}

By merging the code and its opposite in the above construction  we obtain.
\begin{construction}\label{Constr_H2n2distanceanti}
The code $\bigcup _{a=0,2}\{x:x_i \mbox{ mod } 4=a, i\in \{1,\ldots,n\} \}$ is an all-nulls CRC in $G_n$ with covering radius $\rho=n$ and  parameter matrix 

$$\left(%
\begin{array}{cccccc} 0& 2n & 0 & 0 & \ldots & 0 \\
1& 0 & 2n-1 & 0 & \ldots & 0 \\
0& 2 & 0 & 2n-2 & \ldots& 0  \\
&&&\ldots&\\
0& &\ldots & n-1 & 0 & n+1  

\\
0&  &\ldots & 0 &2n & 0 
\\\end{array}%
\right)$$
\end{construction}

We finish with an example of an all-ones CRC. If an all-nulls CRC in $H(2n-1,2)$ with parameter matrix $A$ exists, then  the code $$\{(x_1,\ldots,x_{2n-1},a):x \in C, a\in\{0,1\}\}$$  is a CRC in $H(2n,2)$ with the parameter  matrix $A+E$ \cite{FDF07}. For $C$ being a singleton of $H(2n-1,2)$ we obtain the following from Construction \ref{Constr_H2n2}.

\begin{construction}
For any $n$ there is  all-ones CRC in $G_n$ with $\rho=2n-1$ and  parameter matrix $$\left(%
\begin{array}{cccccc} 1& 2n-1 & 0 & 0 & \ldots & 0 \\
1& 1 & 2n-2 & 0 & \ldots & 0 \\
0& 2 & 1 & 2n-3 & \ldots& 0  
\\
\ldots
\\
\ldots
\\
0& 0 &\ldots & 0 &2n-1 & 1 
\\\end{array}%
\right).$$
\end{construction}



\section{Local structure of $n$-dimensional infinite grid}

Let  $u \in {\bf Z}^n$ be a word of weight not more than $4$. We say $u$ is of type  $abcd$  if  
the sequence of the absolute values of symbols of $u$ is $abcd$ in descending order. For example, the word $(1000\text{-}2 1 )$ is of type $2110$. The type of a word $u$ is called Lee composition vector of $u$ in   \cite{Tar87}.  Clearly, the sets of the words having  fixed type and weight are   orbits of the stabilizer of all-zero word in the automorphism group of the graph $G_n$. The weight two words  are only of
types $2000$ and $1100$. Weight three and four words are split into types $3000$, $2100$,  $1110$ and  types $4000$, $3000$, $2110$, $2200$, $1111$ respectively. 

 Let a cap  be the set of all weight four words
with absolute value $3$ in a fixed position, i.e. caps are the following sets: $$\{3e_i\pm e_j: j \in \{1,\ldots, n\}\setminus \{i\}\}, 
\hphantom{-ai=1,\ldots,n.}$$ 
$$\{-3e_i\pm e_j: j \in \{1,\ldots, n\}\setminus \{i\}\}, \ i=1,\ldots,n.$$ 
We see that the set of all type $3100$ words is the disjoint union of all caps.

Define  the interval graph $I(x)$ of a vertex $x$ as the  subgraph of $G_n$ induced by the following vertex set  
$\{y: d(y,x)+d(y,{\bf 0})=d(x,y)\}$.  Up to types, the interval subgraphs for all weight four  vertices are depicted on Fig.1.

$\input{Itype4000}$
$\input{Itype22_rtd}$

$\input{Itype31}$
$\input{Itype211}$

$\input{Itype1111}$ 

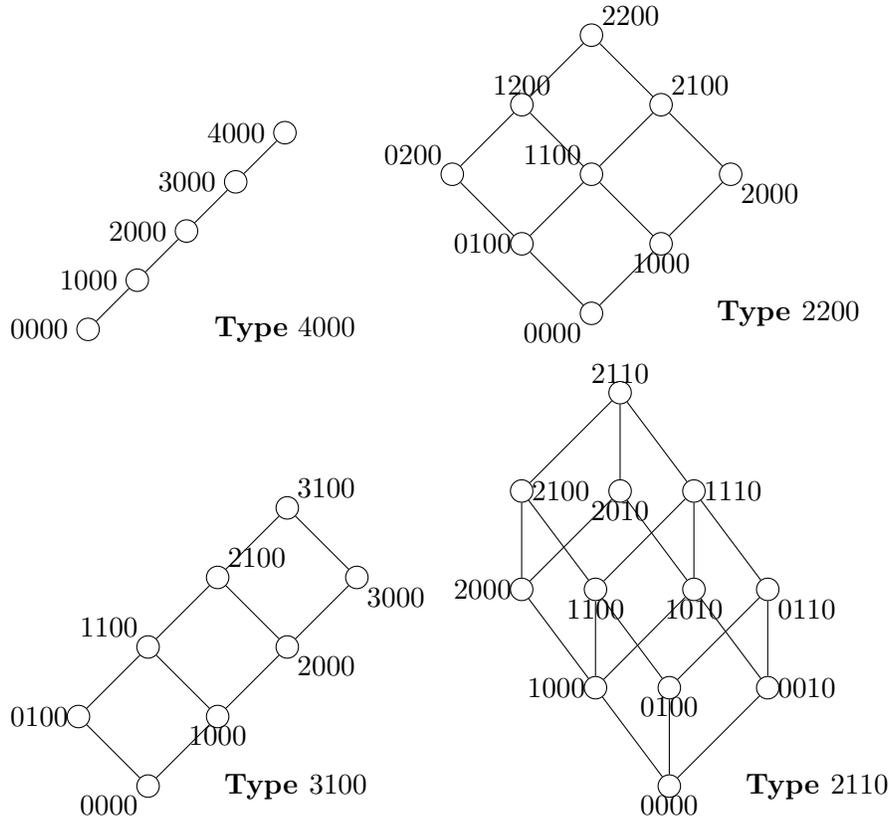
$$$$
 
\begin{tikzpicture}
\node[anchor=base west] at (3, 1) {Fig. 1. \textbf{The interval graphs for weight four words.}};
\node[anchor=base west] at (3, 0.4) {\textbf{The last $n-4$ all-zero symbols are omitted in the figures}};
\node[anchor=base west] at (3, -0.2) {\textbf{throughout the paper for simplicity.}};
\end{tikzpicture}

\section{Null CRCs with $c_2\leq 3$}

In the current section classes of $2$-null and $3$-null CRCs are considered.
We start with  $2$-null CRCs such that $c_1<c_2\leq 3$ (note that $c_1$ is necessarily not greater than $c_2$ due to Theorem \ref{T_AV_1}). Throughout the current and next two sections we  assume that the all-zero word ${\bf 0}$ is in $C$.
\begin{remark} 1.  Since ${\bf 0}$ is in $C$, when $a_{0}=0$, the weight one words are all in $C_1$. If, moreover, $a_{1}=0$, then weight two words are a subset of $C\cup C_2$. The condition $c_1=1$, $a_{0}=a_{1}=0$ is 
equivalent to the minimum distance of $C$ being at least  four.  Since the minimum distance of any CRC in $G_n$ is not greater than four \cite[Theorem 4]{AvgVas22} it is exactly four.

2. Note that if a weight four word $x$ is in a $2$-null  CRC $C$, then, given that $a_0=a_1=0$, all weight one and three words in the interval graph of $x$ are in $C_1$ and weight two words are in $C$ or $C_2$.   
\end{remark} 

In case of  a $2$-null  CRC $C$, any type $1100$ word is adjacent to two words of weight one, which are in $C_1$, we obtain the following.

\begin{lemma}\label{l_a21-2}
 If $C$ is a $2$-null CRC in $G_n$, $n\geq 2$ with $\rho\geq 2$  then any type $1100$ word in $C_2$ is adjacent to exactly $c_2-2$ weight three words in $C_1$.  
\end{lemma}

\begin{corollary}\label{coro_1} Any $2$-null CRC in $G_n$, $n\geq 1$ with $\rho\geq 2$, $c_1=1$ and $c_2=2$ is obtained by Construction \ref{Constr_H2n2} from a CRC in $H(2n,2)$.
\end{corollary}
\begin{proof}

Note that since $c_1=1$, all  weight two and three words in the  interval graph of any word of weight four from $C$ are in $C_2$ and $C_1$  respectively.  
 

In $G_n$, any  word $x$ of type $2000$ e.g. $(20\ldots0)$ (which is in $C_2$) is adjacent to only one word of weight one, $(10\ldots0)$ (which is in $C_1$).  Since $c_2$ is $2$ by Lemma \ref{l_a21-2}, $x$ is at distance two from exactly one word $y$ and by Remark 2.1 $y$ cannot be of weight two, so it is of weight four. The word $y$ cannot be of types $3100$ or $2110$.
Indeed, from Fig.1  the interval graphs of these words contain a word of type $1100$ which by Remark 2.1 is in $C_2$. In the interval graphs of types $3100$ or $2110$, this word is adjacent to at least three words of weights one and three. Since, by Remark 2.2, all these words are in $C_1$, this contradicts the parameter $c_1=2$.  We conclude that the word $y$ is $(40\ldots0)$ of type $4000$. Extending this argument to all type $2000$ words  we conclude that the weight four words in $C$ are exactly all type $4000$ words. Thus, we see that the code has period four and is obtained by Construction \ref{Constr_H2n2}. 

\end{proof}

\begin{lemma}\label{l_interval}
 Let $C$ be a $2$-null CRC in $G_n$, $n\geq 2$ with $\rho\geq 2$, $c_1\leq 2$ and $c_2\geq 3$. Then 
 any word of type $1100$ in $C_2$ is in the interval graph of a weight four word in $C$.
\end{lemma}
\begin{proof}
By Lemma \ref{l_a21-2}, any type $1100$   word in $C_2$ is adjacent to a word $z$ of weight three in $C_1$.  If $z$ is adjacent to a word of weight four in $C$, we are done, so we suppose that $z$ is adjacent   only to weight two words in $C$. By Remark 2.1, the case $c_1=1$ implies that the minimum distance of $C$ is four, a contradiction. So, $c_1$ is $2$ and $z$ is adjacent to exactly two words of weight two in $C$. Up to isomorphism of $G_n$, the word $z$ is $(3000\ldots0)$, $(2100\ldots0)$ or  
 $(1110\ldots0)$. The word $(3000\ldots0)$ is  adjacent to only one weight two word in $G_n$, so it is not $z$. The word   $(2100\ldots0)$ is adjacent to only two words of weight two:
   $(1100\ldots0)$ and 
     $(2000\ldots0)$ and we see that $(1000\ldots 0)\in C_1$ is adjacent to three words in $C$, a contradiction with $c_1=2$. Similarly, in the case when $z$ is $(1110\ldots0)$, up to permutation of three first positions, the weight two neighbors of $z$ are    $(1100\ldots0)$ and    $(1010\ldots0)$, giving the same contradiction with $c_1=2$ for the word $(10\ldots0)$ in $C_1$.

\end{proof}

\begin{lemma}\label{l_main}  Let $C$ be any $2$-null CRC in $G_n$, $n\geq 2$,  with   $\rho\geq 2$, $c_1\leq 2$ and $c_2=3$. Then the following holds:

1. Any cap contains not more than one word in $C$. In particular, there are not more than $2n$ type $3100$ words in $C$.

2. Any type $1100$ word in $C_2$
 is at distance two from exactly one word in $C$ of weight four, which is necessarily of type $3100$.

\end{lemma}
\begin{proof}
1. The statement is obvious for the case $c_1=1$ because by Remark 2.1 the minimum distance of $C$ is four and the distance between any distinct words in any cap is two. 

Let $c_1$ be $2$. Suppose the opposite, up to an automorphism of $G_n$, we have two cases.

Case A. Let the words $(3100\ldots 0)$
and $(3\text{-}10\ldots 0)$ be in $C$. Therefore,
 the words $(2100\ldots 0)$, $(3000\ldots 0)$ and $(2\text{-}100\ldots 0)$ are in $C_1$ because $C$ is a $1$-null code.

We see that the word $(200\ldots0)$, which is in $C$ or in $C_2$ due to Remark 2.2, has four neighbors from $C_1$, namely $(100\ldots0)$, $(2100\ldots 0)$, $(3000\ldots 0)$ and $(2\text{-}100\ldots 0)$. Therefore, since $c_2=3$, the word $(200\ldots0)$ having four neighbors in $C_1$ cannot be in $C_2$ and it is in $C$. However, the word $(3000\ldots 0)\in C_1$ is adjacent to three words of $C$, $(310\ldots0)$, $(3\text{-}10\ldots0)$ and $(200\ldots0)$, which contradicts $c_1= 2$. 

$\input{Lemma_contradiction}$

Case B. Let the words $(3100\ldots 0)$
and $(301\ldots 0)$ be in $C$. Similarly to case A, we see that the word $(200\ldots 0)$ is adjacent to $(100\ldots0)$, $(2100\ldots 0)$, $(3000\ldots 0)$ and $(2010\ldots 0)$ and therefore it is in $C$ , which implies the same contradiction for $(3000\ldots 0)\in C_1$ as in  Case A.

2. Consider a word  $x$ of type $1100$ in $C_2$. Since $c_2=3$, by Lemma \ref{l_interval} the word $x$ is at distance two from exactly one weight four word $y$ in $C$. By examining  the interval graphs on Fig. 1, we see that $y$ is of type $3100$ or $2110$. Indeed, any type $1100$ word in the other interval graphs is adjacent to two words of weight three, which   are in $C_1$ by Remark 2.2, providing contradiction to Lemma \ref{l_a21-2}. Let $y$ be of type $2110$, without restricting the generality, 
it is  $(2110\ldots0)$ and we seek a  contradiction. Consider the interval graphs of type $2110$ word $y$ in Fig. 3. We see that  $x\in C_2$ is  already adjacent to three words from $C_1$ in the interval graph of $y$ and $(1100\ldots 0)$ and $(1010\ldots 0)$ are adjacent to four words from $C_1$ in $I(y)$. We obtain that $x$ can only be $(0110\ldots 0)$.

Case $c_1=1$. Then the minimum distance of $C$ is four. We immediately obtain a contradiction  since two more type $1100$ words: $(110\ldots 0)$, $(101\ldots 0)$  are in the interval graph of $y$ and they belong to $C_{2}$, contradicting $c_2=3$, see Fig. 3.

When $c_1$ is $2$ by Remark 2.2 there are weight two words from $C$  and we see that both $(1100\ldots 0)$ and $(1010\ldots 0)$ must be in $C$ (otherwise we meet the same contradiction as in case $c_1=1$).  
However, in this case,  $(1000\ldots0)\in C_1$ is adjacent to three words in $C$ and we obtain a contradiction with $c_1=2$ (see Fig. 4).

\end{proof}
 \begin{theorem}\label{th_1} A $2$-null CRC in $G_n$, $n\geq 1$ with $\rho\geq 2$, $c_1=1$ and $c_2=3$ exists if and only if it is obtained by Construction \ref{Constr_halved_pc} for $n=2$. There are no $2$-null CRCs in $G_n$ with $c_1=2$, $c_2=3$  for any $n$. 
\end{theorem}

\input{lemma_211contradiction.tex}

\input{lemma_contr2112.tex}

\begin{proof}  We show that there are at least $4(^n_2)-n$ type $1100$ words  in $C_2$. 

We have $4(^n_2)$ words of type $1100$ in $G_n$ and for $c_1=1$ all of them are in $C_2$. For $c_1=2$ let 
 $\beta$ of them be in $C$. By double counting of the edges between these words and the weight one words, we obtain that
$\beta\leq n$.
We see that there are at least $4(^n_2)-n$ type $1100$ words in $C_2$. By Lemma \ref{l_main}.2 the number of these words  
 is not more than the number of type $3100$ words in $C$. As by Lemma \ref{l_main}.1, there are not more than $2n$   type $3100$ words in $C$, 
 we obtain that 
 $4(^n_2)-n\leq 2n$, which holds only for $n\leq 2$. 
 
 The case $n=1$ for  $c_2=3$  is not possible. All parameter matrices of CRCs in $G_2$ were found in \cite{ASV12} and there are no codes with $a_{0}=a_{1}=0$ and $c_1=2$, $c_2=3$ and only one parameter matrix with $a_{0}=a_{1}=0$ and $c_1=1$, $c_2=3$, namely $$\left(%
\begin{array}{cccc} 0& 4 & 0 & 0 \\
1& 0 & 3  & 0 \\
0&3& 0 & 1  \\
0& 0&4 & 0 \\\end{array}%
\right),$$ and the respective CRCs are obtained by Construction \ref{Constr_halved_pc}.

\end{proof}

\begin{theorem}
 The only $3$-null CRCs in $G_n$  $n\geq 1$ with $\rho\geq 2$ and  $c_2\leq 3$   to exist are:

1. CRCs in $G_1$ with any $\rho\geq 2$, $a_{0}=a_{1}=0$, $c_{i}=b_{i}=1$, $i=1\ldots,\rho-1$, $a_{\rho}=0$ and 
$\rho\geq 3$, $a_{0}=a_{1}=0$, $c_{i}=b_{i}=1$, $i=1\ldots,\rho-1$, $a_{\rho}=1$.

2. The CRCs obtained from these  codes in $G_2$ with $c_1=c_2=2$ and in $G_3$ with $c_1=c_2=3$ by Construction \ref{Constr_ZMult}. 
 
 3. CRCs with $c_1=1$, $c_2=2$  obtained  by Construction \ref{Constr_H2n2}.
 
4. CRCs with  $c_1=1$, $c_2=3$  obtained by Construction \ref{Constr_halved_pc} for $n=2$.

\end{theorem}

\begin{proof}
The parameters of CRCs in $G_1$ were obtained in \cite{ASV12} and the only $3$-null codes in $G_1$ are series with parameters as in the first and second statement of Theorem.

   By Theorem \ref{T_AV_1}, the sequence of elements below the main diagonal of parameter matrix is  not decreasing.   
    If $c_1=c_2$ and $\rho\geq 3$ then because $a_{1}=a_{2}=0$, we have $b_1=b_2$  and then  due to Theorem \ref{T_AV_2} this implies that $C$ in $G_n$ is obtained from  Construction \ref{Constr_ZMult}  applied to a CRC in $G_1$, where $n$ is a divisor of all entries of parameter matrix.     For $c_1=c_2=2$ or $3$ this implies that $n$ is $2$ and $3$ respectively.   
    If $c_1=c_2$ and $\rho=2$ then, since $a_{1}=a_{2}=0$ we obtain that $c_1=c_2$ is equal to graph valency $2n$, which cannot occur.

The case $c_1<c_2$ follows from Corollary \ref{coro_1} and Theorem \ref{th_1}. This case  corresponds to the third and fourth statements of Theorem.  
    
\end{proof}

\section{$1$-null CRCs with  $a_{1}=1$}

\begin{theorem}\label{T_1null}
Any CRC in $G_n$, $n\geq 1$  with $\rho\geq 1$, $a_{0}=0$ and $a_{1}=1$ has $c_1=1$ and is obtained by Construction \ref{Constr_Hn3}.   
\end{theorem}
\begin{proof}
Since $a_0=0$, all words of weight one are in $C_1$, each of which has exactly one weight two neighbor in $C$ because $a_1=1$. Moreover, if there is a word of type $1100$ as such a neighbor,  then $a_{1}>1$, a contradiction. We conclude that all words of type $2000$ are also in $C_1$. By the same arguments, all words of type $3000$ are in $C$ and $C$ has  period three in any direction and therefore $C$ arizes from a  CRC in $H(n,3)$ by Construction \ref{Constr_Hn3}.

Consider the respective code  in $H(n,3)$. The graph induced by the ball of radius $1$ in $H(n,3)$ is a union of $n$ cliques of size $3$ any pair of which has only one common vertex, the center of the ball. In case when the center of the ball is in $C_1$, because $a_{1}=1$, only one of the cliques, say $K$, contains an extra vertex of $C_1$. The remaining cliques contain words of $C$ and $C_2$, but because  $a_{0}$ is $0$, actually only words of $C_2$. We are left with the case where there is only one vertex in the ball from $C$ in the clique $K$, thus $c_1=1$.  
\end{proof}
An obvious example of ternary CRC fulfilling  conditions of Theorem \ref{T_1null}  is  a singleton vertex code. 
Furthermore, by studying the database \cite[Tables 7.9-7.11]{KKM25} of completely regular codes in $H(n,3)$ we found that the only known CRCs that meet the conditions of Theorem \ref{T_1null} are the ternary Golay code and its related codes \cite{BRZ19}: punctured, shortened and extended codes. 

\section{All-ones CRC}
In this section we study all-ones CRCs. In this case, each connected component of $C_i, i=0,\ldots ,\rho$, consists of two adjacent words. 
Throughout this section we assume that the all-zero word ${\bf 0}$
and the word $-e_n$ are in $C$.

We say that the following set of words in the graph $G_n$ 
$$\{(x_1,\ldots,x_n) :
\ x_1,\ldots,x_{n-1}\in {\bf Z} , x_n\in\{2s,2s-1\}\}$$
is the $s$-slice. We call a collection of words $x$ and $x-e_n$ in $s$-slice a slice pair. An $n$-line is a set of all words in $G_n$ that differ only in the $n$-th symbol. We now show that any $n$-line in $G_n$ is a collection of slice pairs, whose words are at alternating distances from an all-ones CRC $C$, forming domino-like structure, see Fig. 5.

\begin{lemma} \label{all-one}
Let $C$ be an all-ones CRC in $G_n$, $n\geq 1$. The following holds. 

1. Two words in any slice 
pair are in $C_i$, for some $i\in \{0,\ldots,\rho\}$.

2. If an $n$-line contains words of $C_i$, then the words from  $C_i$ in the direction $n$ have period four and are split into slice pairs. In particular, any $n$-line consists of slice pairs from $C_i$, $C_{i+1}$, for some $i \in \{0,\ldots,\rho-1\}$. 

\end{lemma}

\begin{proof}

Firstly, we note that the existence of a slice pair in $C_i$ for some $i\in\{0,\ldots,\rho\}$  implies that the words in the slice are split into slice pairs at the same distance from $C$.
Indeed, let $x$ and $x-e_n$ be a slice pair in the $s$-slice for some $s$ in ${\bf Z}$, both from $C_i$, for some $i\in\{0,\ldots, \rho\}$. Consider the cycles of length four containing the slice pair:   
$x$, $x-e_n$, $x+ae_j$, $x+ae_j-e_n$, $a\in \{-1, 1\}$, $j\in\{1,\ldots,n-1\}$, where all words are in the $s$-slice. The words $x+ae_j$, $x+ae_j-e_n$ are not in $C_i$ because $x$ and $x-e_n$ are in $C_i$ and $C$ is an all-ones CRC. Then they are both in $C_{i-1}$ or in $C_{i+1}$. Extending this argument for  $x+ae_j$, $x+ae_j-e_n$ instead of $x$ and $x-e_n$ we obtain the required.

\input{Dominos}

1. By convention, ${\bf 0}$
and $-e_n$ are in $C$, so by the above the $0$-slice is split in to slice pairs with both words at the same distance from $C$. 

Because CRC $C$  has the parameter  $a_0=1$ the following neighbors of ${\bf 0}$ and $-e_n$:  $ae_i, ae_i-e_n$, $i\in\{1,\ldots,n-1\}$, $a \in \{-1, 1 \}$ are in $C_1$. Furthermore, because $C$ has $a_1=1$, $e_n$
and $-2e_n$ from $C_1$ have unique neighbors in $C_1$. We show that they are   $2e_n$ and  $-3e_n$ respectively. Indeed, suppose the opposite and there is a neighbor of  $e_n$,   $e_n+ae_i$, $a\in \{-1,1\}$, $i\in \{1,\ldots, n-1\}$ which is in $C_1$. Then  $e_n+ae_i$   in $C_1$ is adjacent to two words in $C_1$, namely $e_n$ and $ae_i$, a contradiction with the parameter $a_1=1$ of CRC $C$. Therefore, the neighbor of $e_n$ in $C_1$ can only be $2e_n$. Similar argument for $-2e_n$ shows that $-3e_n$ is in $C_1$. By  the remark at the beginning of the proof we see that the words in any slice pairs in the $1$-slice and in $-1$-slice are at the same distance from $C$.

We extend this argument in  $s$-slices for all $s\in {\bf Z}$  to obtain the required.

2. Obviously, any word in a slice pair $x$ and $x-e_n$ from $C_i$ is adjacent to the same number of words from $C_{i-1}$ ($C_i$ and $C_{i+1}$ respectively) in the $s$-slice. By the first statement of Lemma,  
$x+e_n$, $x+2e_n$, $x-2e_n$ and $x-3e_n$,  are in $C_{i-1}$ or $C_{i+1}$ simultaneously. The same argument applied to   
$x+e_n$, $x+2e_n$ implies the period of four for $C_i$ in any $n$-line. 
\end{proof}

From Lemma \ref{all-one}.2 there are two code edges at distance 3 in any $n$-line with a code  word. We now compute the minimum distance between code slice pairs in any slice.
\begin{lemma}\label{l_distance_slice}
Let $C$ be all-ones CRC in $G_n$, $n\geq 2$ with $\rho\geq 1$ and $c_1=1$. Then the minimum distance between slice pairs from $C$ in the same slice is at least $4$.

\end{lemma}
\begin{proof}
Let $d$ be the minimum of the distances between $x$ and $y$, for all $x\subset A,y\subset B$ and all  distinct slice pairs $A$, $B$ from $C$ in the same slice.  By Lemma \ref{all-one}.1 we see that  $d\neq 1$.  Considering the shortest path between the closest $x$ and $y$ in $C$, we see that if $d=2$, then parameter $c_1\geq 2$ and if $d=3$ then $a_{1}\geq 2$, a contradiction.

 \end{proof}

\begin{remark} Let $C$ be an all-ones CRC in $G_n$ with $c_1=1$. 

1.  Note that no weight two word is in $C$. Otherwise, there is a weight one word adjacent to such a word and all-zero word. This  weight one word cannot be in $C$ or $C_1$ because $a_0=c_1=1$.

2. Consider the interval graph of weight four word $y$ from $C$. From the above, taking into account that  ${\bf 0}$ is in $C$ if type $1100$ word is in $C_2$, all its neighbors  in the interval graph of $y$ are from $C_1$.
\end{remark}
The lemmas below follow the narrative of Section 5.

\begin{lemma} \label{all-one-12} 
Any all-ones CRC in $G_n$, $n\geq 1$ with $\rho\geq 2$, $c_{1}=1$ and $c_2=2$ is obtained by Const\-ruction \ref{Constr_H2n2}.    
\end{lemma}
\begin{proof}  Firstly, we note that all type $2000$ words in the $0$-slice are in $C_2$. Indeed, by our convention $ae_j$ is in $C_1$ for all $a\in\{-1,1\}$, $j\in \{1,\ldots,n-1\}$ and by Lemma \ref{all-one}.1 the slice pair  $ae_j$, $ae_j-e_n$ is in $C_1$. The type $2000$ word $2ae_j$ cannot be in $C_1$ because otherwise $ae_j\in C_1$ is adjacent to two words in $C_1$, contradicting $a_1=1$. 
Moreover, by Remark 3.1, $2ae_j$ is not in $C$, so all type $2000$ words in the $0$-slice are in $C_2$.

Thus, any word from the $0$-slice of type $2000$ is in $C_2$, which  given that $c_2=2$, implies it  must be at distance two from a codeword $y$ different from ${\bf 0}$. By Remark 3.1, $y$ cannot be of weight two, so it is of weight four.  We observe that $y$ is of type $4000$ (in the $0$-slice) or of types $2200$, $3100$, $2110$ (in the same slice or not).  We will show that the last three cases are impossible. 

Consider the interval graph $I(y)$, $y\in C$. Since ${\bf 0}$ and $y$ are in $C$, all weight one and three words in the graph $I(y)$ are from $C\cup C_1$. Due to Remark 3.1, the words of weight two in the graph $I(y)$ are in $C_1\cup C_2$.  Similarly to Corollary \ref{coro_1}, in case when $y$ is of types $2200$,  $3100$, $2110$, we see from Fig. 1. that there is always a type $1100$ word $z$ in $I(y)$ having at least three neighbors in $I(y)$. 
If the word $z$ is in $C_2$, by Remark 3.2, moreover, all neighbors  of $z$ in $I(y)$ are in $C_1$, providing contradiction to $c_2=2$. If $z$ is in $C_1$,  by pigeon hole principle at least two its neighbors are either in $C$ or in $C_1$, providing contradiction to $a_1=c_1=1$.

Finally, we conclude that any word of type $4000$ in the $0$-slice is in $C$. Since the period of $C$ is four in the $n$th direction by Lemma \ref{all-one}.2, this implies the period four for $C$ in all $n$ directions.
\end{proof}

\begin{lemma}\label{one-13-23}  Let $C$ be an all-ones CRC in $G_n$, $n\geq 2$  with $\rho\geq 2$, $c_1=1$ and $c_2=3$. Then the following holds:

1. The intersection of any cap with the $0$-slice contains at most one word in $C$.  
In particular, there are not more than $2(n-1)$ type $3100$ words of $0$-slice in $C$.

2. Any type $1100$ word $x$ with $x_n=0$ is in $C_2$ and is  at distance two from exactly one type $3100$ word $y$ in $C$ such that $y_n=0$. \end{lemma}
\begin{proof}
   1. The statement  holds  because by Lemma \ref{l_distance_slice} the minimum distance between slice pairs in any slice is at least four and the distance between any distinct words in any cap is two.

2.  Let $x$ be a type $1100$ word  with $x_n=0$.  Both of its weight one neighbors are in $C_1$, because the only weight one codeword is $-e_n$. We conclude that $x$ cannot be in $C_1$ because $a_1$ is $1$. It also cannot be in $C$ by Remark 3.1 and
so it is in $C_2$.  
Since $c_2$ is $3$,  the word  $x$ is at distance two from exactly one nonzero  word $y$ in $C$. By Remark 3.1 there are no codewords of weight two, so $y$ in $C$ has weight four. 

Case A. Let the word $y$  be of types $2200$, $2110$ or $1111$. As in the proof of Lemma \ref{all-one-12}, the weight one and three words from $I(y)$ are in $C\cup C_1$ and weight two words from $I(y)$ are in $C_1 \cup C_2$.

From Fig. 1, the interval graph $I(y)$   always contains a word $z$ of type $1100$ having four neighbors  in the graph $I(y)$. If $z$ is in $C_2$, by Remark 3.2, all four  weight one and three neighbors of $z\in C_2$ in $I(y)$ are in $C_1$, contradicting $c_2=3$.
If $z$ is in $C_1$ at least two its neighbors in $I(y)$ are in $C$ or $C_1$, contradicting $a_1=c_1=1$.

Case B. Let $y$ be of type $3100$ at distance two from type $1100$ word $x$, $x_n=0$. Then obviously $y_n=0$, $y$ is in the $0$-slice and we obtain the required.
\end{proof}

\begin{theorem}\label{th_all-one} An all-ones CRC  in $G_n$, $n\geq 1$ with $\rho\geq 2$,  $c_1=1$  and $c_2=3$ exists if and only if  $n$ is $2$.  
\end{theorem}
\begin{proof}   

We have $4\binom{n-1}{2}$ words of type $1100$ with the zero last symbol and by Lemma \ref{one-13-23}.2 all of them are in $C_2$.  By Lemma \ref{one-13-23}.2 there is a bijection between these words and type $3100$ codewords with the zero last symbol. Since by Lemma \ref{one-13-23}.1 there are at most $2(n-1)$ such words,
 we conclude that 
 $4\binom{n-1}{2}\leq 2(n-1)$, 
 i.e., $n\leq 2$. The case $n=1$ and $c_2=3$ is impossible and there is only one all-ones parameter matrix of the CRCs in $G_2$ with $c_1=1$ and  $c_2=3$ \cite{ASV12}, namely  $\left(%
\begin{array}{ccc} 1& 3 & 0  \\
1& 1 & 2 \\
 0 & 3 & 1 
\\\end{array}%
\right).$    
\end{proof}

\section{Binary linear programming for null CRC in $G_3$ and $G_4$}

All parameter matrices for completely regular codes in $G_2$ were enumerated in \cite{ASV12}, so the first open cases occur for $n = 3$ and $n = 4$. These results are consistent with the characterization of parameters of  completely regular codes with covering radius $2$ in \cite{Tarasov}. The discussion below focuses on CRCs in the infinite rectangular grid, but we also state the key points for arbitrary vertex-transitive graphs with an infinite number of vertices. 
   
The "perfect coloring"\, equation for a completely regular code $C$ with parameter matrix $A$ in a finite graph of order $v$ with adjacency matrix $M$ can be formulated as a classical $0$-$1$ linear programming problem.
We associate the elements of the distance partition $C_0 = C, \ldots, C_\rho$ with their characteristic vectors $\chi_0, \ldots, \chi_\rho$.
The existence problem for a completely regular code thus becomes the following $0$-$1$ linear programming problem. 
$$\chi_0, \ldots, \chi_{\rho} \in \{0,1\}^{v},\chi_0 + \ldots + \chi_{\rho} = {\bf 1},M(\chi_0\ldots \chi_{\rho})=(\chi_0\ldots \chi_{\rho})A.$$ 
We now outline two main differences from the original integer linear program\-ming approach discussed above.

Firstly, since $G_n$ (or, in general, vertex-transitive graph $\Gamma$) is a graph with infinitely many vertices, the search space must be restricted to a certain finite subgraph. In previous work, a computer-based enumeration of perfect colorings with $9$ colors \cite{Kro9} was conducted using a finite subgrid as the search domain.
During our testing, we found that the results for our codes using linear programming are more accurate (i.e., more parameter matrices are excluded) when the subgraph is induced by a ball.

Let $M$ denote the adjacency matrix of the subgraph $\Gamma'$ of $G_n$ of order $v$, induced by the vertices in $G_3$ or $G_4$ of weight at most six.
We denote by $J$ the set of all vertices of $\Gamma'$ that have degree $6$ or $8$ in $\Gamma'$ (i.e., the valencies of $G_3$ and $G_4$, respectively).
Let $P_J$ be the binary diagonal matrix of size $v$ such that $(P_J)_{x,x} = 1$ if and only if $x \in J$.
Let $E$ denote the identity matrix of size $v$.

We now express the above considerations in matrix form as follows:
 \[
\text{{$LP'$: }}
\quad
\left\{
\begin{alignedat}{2}
&\chi_0, \ldots, \chi_{\rho} \in \{0,1\}^{v},\chi_0 + \ldots + \chi_{\rho} = {\bf 1},\\
& P_J  M(\chi_0 \ldots \chi_{\rho}) = P_J  (\chi_0 \ldots \chi_{\rho})A,\\ 
&(E - P_J)  M(\chi_0 \ldots \chi_{\rho}) \leq (E - P_J)  (\chi_0 \ldots \chi_{\rho})A,\\&\chi_0[0] = 1.
\end{alignedat}
\right.
\]
 
 Here, we index all vertices of $\Gamma'$ with indices $0,\ldots, v-1$ and due to the vertex transitivity of the graph $G_n$, we assume the zeroth variable (e.g., the all-zero word) to be in $C = C_0$, or equivalently, $\chi_0[0] = 1$. In general, the respective linear programming problem $LP'$ for a CRC with a given parameter matrix can be considered for any  subgraph $\Gamma'$ of any vertex-transitive graph.
If the linear programming problem $LP'$ is infeasible, then a CRC with parameter matrix $A$ does not exist in $G_n$ (vertex transitive graph $\Gamma$). On the other hand, not that the feasibility of the problem $LP'$ does not imply  the existence of a completely regular code in $\Gamma$.

The second feature is added to eliminate several parameter matrices (pos\-sibly infinitely many, depending on the graph) for CRCs in infinite graphs.
The covering radius of CRCs in $G_n$ can be relatively large (up to $2n + 1$ for irreducible CRCs \cite{AvgVas22}).
To eliminate several CRCs that share the same submatrix of parameter matrix, we propose the following relaxation.
Let $B$ denote the upper-left $\left(\rho+1\right) \times \rho$ submatrix of parameter matrix $A$ of the CRC.
We introduce additional variables by defining the sum of all $\chi_i$ for $i \geq \rho + 1$ as $F$  in the following binary integer linear programming problem.
\[
\text{{$LP_{\geq \rho}$: }}
\quad
\left\{
\begin{alignedat}{2}
&\chi_0, \ldots, \chi_{\rho}, F \in \{0,1\}^{v},
\chi_0 + \ldots + \chi_{\rho} +F= {\bf 1}, \\
&P_J  M(\chi_0 \ldots \chi_{\rho-1}) = P_J  (\chi_0 \ldots \chi_{\rho})B, \\
&(E - P_J)  M(\chi_0 \ldots \chi_{\rho-1}) \leq (E - P_J)  (\chi_0 \ldots \chi_{\rho})B, \\
&\chi_0[0] = 1,
\\
&\sum_{x=0}^{v-1} F[x] \geq  0.
\end{alignedat}
\right.
\]

Variations of the problem $LP_{\geq \rho}$ are obtained by replacing the last equation with either $\sum\limits_{x=0}^{v-1} F[x] = 0$ or $\sum\limits_{x=0}^{v-1} F[x] \geq 1$. We denote these modified problems by $LP_{=\rho}$ and $LP_{>\rho}$, respectively.

Note that the problem $LP_{=\rho}$ for a matrix $B$ is equivalent to a coloring of the ball using exactly $\rho+1$ colors such that the coloring is consistent with  parameter matrix obtained from $B$ by appending a column of row sums to the right. Below, we refer to this augmented matrix as the extension of $B$.

We summarize the above discussion as follows. 
\begin{proposition}
Let $B$ be a $(\rho+1) \times \rho$ matrix and let $M$ be the adjacency matrix of a subgraph $\Gamma'$ of a vertex-transitive graph $\Gamma$ with $v$ vertices and valency $t$. Let $J$ be the set of vertices in the subgraph $\Gamma'$ that have valency $t$ and let $P_J$ be the $v \times v$ diagonal matrix defined by $(P_J)_{x,x} = 1$ if $x \in J$ and $0$ otherwise. Then the following holds:

1. If the integer problem $LP_{\geq \rho}$ (respectively, $LP_{> \rho}$) is infeasible, then there are no completely regular codes (CRCs) in $\Gamma$ with covering radius at least (respectively, greater than) $\rho$ whose parameter matrix contains $B$ as its upper-left $(\rho+1) \times \rho$ submatrix.

2. If the integer problem $LP_{=\rho}$ is infeasible, then there is no CRC in $\Gamma$ with covering radius exactly $\rho$ whose full parameter matrix is the extension of $B$.
\end{proposition}

If a matrix $B'$ with at least $\rho+2$ rows has a $(\rho+1)\times\rho$ upper-left submatrix $B$, we call $B'$ a descendant of $B$. 
If the problem $LP_{\geq \rho}$ is infeasible for the matrix $B$, then it is also infeasible for all descendants of $B$ and no CRCs in $G_n$ exist with covering radius at least $\rho$ whose parameter matrix has $B$ as its upper-left submatrix. In the subsequent considerations, we restrict descendants to those satisfying the ascending and descending conditions for the entries above and below the diagonal, as described in Theorem~\ref{T_AV_1}

{\bf Example 1.} Let $B$ be $\left(%
\begin{array}{cc} 0& 6   \\
1& 0  \\
0& 5  \\\end{array}%
\right)$, $n$ is $3$. The extended matrix of $B$ is $\left(%
\begin{array}{ccc} 0& 6 &0   \\
1& 0 & 5 \\
0& 5 &1  \\\end{array}%
\right)$. Taking into account restrictions imposed by  Theorem \ref{T_AV_1}, the $4\times 3$ descendants of $B$ are  $\left(%
\begin{array}{cccc} 0& 6 &0  \\
1& 0 & 5  \\
0& 5 & 0\\
0& 0 & 5  \\\end{array}%
\right)$,  $\left(%
\begin{array}{cccc} 0& 6 &0  \\
1& 0 & 5  \\
0& 5 & 0\\
0& 0 & 6  \\\end{array}%
\right)$, $\left(%
\begin{array}{cccc} 0& 6 &0  \\
1& 0 & 5  \\
0& 5 & 1\\
0& 0 & 5  \\\end{array}%
\right)$,  $\left(%
\begin{array}{cccc} 0& 6 &0  \\
1& 0 & 5  \\
0& 5 & 1\\
0& 0 & 6  \\\end{array}%
\right)$. We see that the last two matrices cannot be submatrices of parameter matrices of CRCs, because of the graph degree is six and third rows, the set $C_3$  is disjoint to $C_2$.

\begin{table}[!h]
\caption{ The parameters of $1$-null CRCs in $G_3$}

\begin{center}
\begin{tabular}{|c|c|}
\hline
Parameters&Construction\\\hline
$[0,6 | 1,5 ]$& A perfect code in $G_3$ \\\hline
$[0,6 | 2,4 ]$&  A CRC in $H(6,2)$ from \\ & a perfect code in $H(3,2)$,\\ &Construction \ref{Constr_ZMult} \\\hline
$[0,6 | 3,3 ]$& A CRC in $G_3$ from a perfect \\& code in $G_1$, Construction \ref{Constr_ZMult} \\
\hline
 $[0,6|6,0]$ & Even weight code in $G_3$\\
 \hline
 $[0,6 | 1,2,3 | 2,4]$& A CRC from triangular  \\
& grid \cite{V22}, Constr. \ref{Constr_Triangular}\\ 
\hline
$[0,6 | 1,2,3 | 3,3]$&  A CRC from triangular \\
& grid \cite{V22}, Constr. \ref{Constr_Triangular}\\ 
\hline
$[0,6 | 1,3,2 | 6,0]$& Constr. \ref{Constr_Hn3} for \\
& repetition code in $H(3,3)$ \\
\hline
$[0,6 | 1,4,1 | 6,0 ]$& Constr. \ref{Constr_H2n2} for shortened\\ &   perfect code  in $H(6,2)$ \\

\hline
$[0,6 | 1,0,5 | 6,0]$ & A diameter perfect codes \\
&in $G_{3,q}$, Constr. \ref{Constr_diam} with $t=1$ \\\hline
$[0,6|2,0,4|6,0]$& Two cosets of  \\
 & diameter perfect codes  in $G_{3}$,  \\
& Constr. \ref{Constr_diam} with $t=2$ \\\hline
 $[0,6|3,0,3|3,3]$& Construction \ref{Constr_ZMult}\\& for a perfect code in $G_1$\\
\hline
$[0,6|2,0,4|4,0,2|6,0]$&Constr. \ref{Constr_H2n2}\\
&+ Constr. \ref{Constr_H2n2distance} in $H(3,2)$\\
\hline
$[0,6 | 1,0,5| 2,0,4|6,0]$& $\{(0\ldots 0),(1\ldots 1)\}$\\
& in $H(6,2)$, Constr. \ref{Constr_H2n2distanceanti} \\
\hline
$[0,6 | 1,0,5 | 5,0,1 |6,0]$
& Even weight subcode\\ 
&  of perfect code in $G_{3}$,\\
&Constr. \ref{Constr_halved_pc}\\\hline
$[0,6 | 1,1,4 | 2,2,2| 3,3]$& Constr. \ref{Constr_Hn3} for\\ & a singleton vertex in $H(3,3)$\\
\hline

$[0,6 | 1,0,5 | 2,0,4| 3,0,3| 4,0,2| 5,0,1|6,0]$& A singleton vertex \\
&in $H(6,2)$, Constr. \ref{Constr_H2n2distance}  \\
\hline
 $[0,6 | 3,0,3 |\ldots | 3,0,3 |6,0]$ for any $\rho\geq 2$& Constr. \ref{Constr_ZMult} for CRCs\\ & in $G_1$ \cite{ASV12}\\
\hline

\end{tabular}
\label{tab1}
\end{center}

\end{table}

\begin{theorem}\label{ThG_3} Up to opposite code, any $1$-null CRC in $G_3$ has one of the parameters from Table 1.





\end{theorem}
\begin{proof}
We split the codes into two main cases:
$\rho=1$ and $\rho\geq 2$.

I. CRCs with $\rho=1$. It is not hard to see that up to opposite code the only CRC with covering radius $1$ and $a_{0}=a_{1}=0$ is the even weight code in $G_3$. The CRCs with $\rho=1$ and $a_{0}=0$, $a_{1}\geq 1$ are studied by problem $LP_{=\rho}$. The problem is feasible only for the following $2\times 1$ matrices  
 $B=\left(%
\begin{array}{c} 0 \\
a \\
\end{array}%
\right)$, $a=1$, $2$, $3$ with extended $2\times 2$ matrices 
 $[0,6|1,5]$, $[0,6|2,4]$, $[0,6|3,3]$, all of which are parameter matrices for CRCs (see Table 1). Alternatively, the  parameters $c_1=4$, $5$ could be excluded based on a necessary condition for CRCs with $\rho=1$ from \cite[Theorem 2]{DGHM09}.

II. $1$-null CRCs with $\rho\geq 2$. First, we make use of our results from previous sections: by Corollary \ref{coro_1} and Corollary \ref{T_1null} all $2$-null CRCs with $c_1=1$, $c_2=2$ are obtained as $2$-null CRCs from the Hamming graph $H(6,2)$ and all $1$-null CRCs with $a_{1}=1$ are obtained from CRCs in the Hamming graph $H(3,3)$. The classification of the latter codes is known, see \cite{KKM25} and the only such codes are singleton vertices in $H(6,2)$, $H(3,3)$ and a pair of words at Hamming distance $6$ in $H(6,2)$   with the parameters:
 $$[0,6 | 1,0,5 | 2,0,4| 3,0,3| 4,0,2| 5,0,1|6,0], [0,6 | 1,1,4 | 2,2,2| 3,3],$$ $$[0,6 | 1,0,5| 2,0,4|6,0].$$

The vast majority of $3\times 2$ matrices are excluded because $LP_{\geq 2}$ is infeasible.

A. $LP_{=2}$ is feasible for $3\times 2$ matrix. It turns out that all corresponding $3\times 3$ extended matrices are indeed the parameter matrices for CRCs. We split between several subcases.

A1. $LP_{=2}$ is feasible, $LP_{>2}$ is infeasible. The corresponding extended  matrices are: 
$$[0,6 | 1,0,5 | 6,0], [0,6 | 1,3,2 | 6,0],  
[0,6|1,4,1|6,0] 
[0,6|2,0,4|6,0], [0,6|3,0,3|6,0].$$ 

A2. $LP_{=2}$ is feasible, $LP_{>2}$ is feasible but the matrix has no feasible  descendants w.r.t $LP_{\geq 3}$. The respective extended matrices from this case are $$[0,6 | 1,2,3 | 2,4], [0,6 | 1,2,3 | 3,3].$$
Both are parameter matrices of CRCs from triangular grid (though other constructions can exist). 

A3. $LP_{=2}$ is feasible, $LP_{>2}$ is feasible and has feasible  descendants w.r.t $LP_{\geq 3}$.  There is only one such matrix, which is $\left(%
\begin{array}{cc} 0& 6   \\
3& 0  \\
0& 3  \\\end{array}%
\right)$. The corresponding extended matrix is  $$[0,6|3,0,3|3,3],$$ which is  parameter matrix of a CRC.

For all $4\times 3$ descendants of the matrix $\left(%
\begin{array}{cc} 0& 6   \\
3& 0  \\
0& 3  \\\end{array}%
\right)$, $LP_{\geq 3}$ is feasible only for the following matrices: $\left(%
\begin{array}{cccc} 0& 6 & 0  \\
3& 0 & 3  \\
0& 3 & 0 \\
0& 0 & 3 \\
\end{array}%
\right), \left(%
\begin{array}{cccc} 0& 6 & 0  \\
3& 0 & 3  \\
0& 3 & 0 \\
0& 0 & 6 \\
\end{array}%
\right).$ Since both matrices have zeros on the main diagonal, the corresponding extended matrices fulfill the conditions of Theorem \ref{T_AV_2} and we have a description of such codes. These codes exist for all $\rho \geq 3$.

B. $LP_{\geq 3}$ is feasible, but $LP_{=2}$ is not. In this case, all $4 \times 3$ descendants that are feasible with respect to $LP_{\geq 3}$ cannot have descendants due to the rightmost bottom element being six. Indeed, since in this case $c_3 = 6$, by Theorem \ref{th_1}, $c_4$ is also $6$, which implies that $C_3$ and $C_4$ are disjoint, leading to a contradiction. Moreover, the corresponding extended matrices are the parameter matrices of existing CRCs. We have three such matrices: $\left(%
\begin{array}{cc} 0& 6   \\
1& 0  \\
0& 5  \\\end{array}%
\right)$, $\left(%
\begin{array}{cc} 0& 6   \\
2& 0  \\
0& 4  \\\end{array}%
\right)$, $\left(%
\begin{array}{cc} 0& 6   \\
4& 0  \\
0& 5  \\\end{array}%
\right),$ feasible descendants are
 $\left(%
\begin{array}{ccc} 0& 6&0   \\
1& 0&5  \\
0& 5 &0 \\
0 &0 &6 \\
\end{array}%
\right)$, $\left(%
\begin{array}{ccc} 0& 6&0   \\
2& 0&4  \\
0& 4 &0 \\
0 &0 &6 \\
\end{array}%
\right)$, $\left(%
\begin{array}{ccc} 0& 6&0   \\
4& 0&2  \\
0& 5 &0 \\
0 &0 &6 \\
\end{array}%
\right)$. The corresponding extended matrices in the compact form are  $$[0,6 | 1,0,5 | 5,0,1 |6,0], [0,6 | 2,0,4 | 4,0,2 |6,0], [0,6 | 4,0,2 | 5,0,1 |6,0].$$ All these matrices are the parameter matrices of some CRCs but  the last  matrix is not new as it the parameter matrix of the opposite code obtained by the above mentioned Construction \ref{Constr_H2n2distanceanti}.

\end{proof}

\begin{theorem}

1. The only $2$-null CRCs in $G_4$ that exist with $c_1=1$, $c_2=2$ and any $\rho \geq 2$ are obtained by Constructions \ref{Constr_H2n2distance} and \ref{Constr_H2n2distanceanti}.

2. There are no $2$-null CRCs in $G_4$ for $c_1=1$, $c_2=3$, $4$, $5$, $6$ and any $\rho \geq 2$.

3. The only $2$-null CRCs in $G_4$ that exist with $c_1=1$, $c_2=7$, $c_2=8$ and any $\rho \geq 2$ are obtained from Construction \ref{Constr_halved_pc} and Construction \ref{Constr_diam} with $t=1$.
\end{theorem}

\begin{proof}
By Corollary \ref{coro_1}, the CRCs with $c_1=1$ and $c_2=2$ correspond to CRCs in $H(8,2)$, which are characterized in \cite{KKM25}. The respective CRCs in $G_4$ are obtained by Constructions \ref{Constr_H2n2distance} and \ref{Constr_H2n2distanceanti}. The CRCs with $c_2=3$ were excluded by Theorem \ref{th_1}. The remaining cases $c_2=4$, $5$, $6$ and $7$ were analyzed similarly to Theorem \ref{ThG_3}, considering problems $LP_{\geq}$, $LP_{=}$ and descendants if necessary. The only non-excluded parameter matrix was the one from Construction \ref{Constr_halved_pc}. The case $c_2=8$ corresponds to the matrix from Construction \ref{Constr_diam} with $t=1$.
\end{proof}


\end{document}

%% file: Itype4000.tex
\begin{tikzpicture}[scale=2] 

\node at (2, 0) {\textbf{Type} 4000 };
\begin{scope}[rotate around={0:(0,0)}]

\node (v0000) at (0,0) [circle, draw, fill=white, inner sep=3pt] {}; 
\node (v1000) at (0.5,0.5) [circle, draw, fill=white, inner sep=3pt] {}; 
Distance 1
\node (v2000) at (1,1) [circle, draw, fill=white, inner sep=3pt] {}; 
\node (v3000) at (1.5,1.5) [circle, draw, fill=white, inner sep=3pt] {}; 
\node (v4000) at (2,2) [circle, draw, fill=white, inner sep=3pt] {}; 

\node[ left] at (v0000) {$0000$\,\,};
\node[left] at (v1000) {$1000$\,\,};
\node[left] at (v2000) {$2000$\,\,};
\node[  left] at (v3000) {$3000$\,\,};
\node[ left] at (v4000) {$4000$\,\,};

\draw (v0000) -- (v1000);
\draw (v1000) -- (v2000);
\draw (v2000) -- (v3000);
\draw (v3000) -- (v4000);

\end{scope}
\end{tikzpicture}

%% file: Itype22_rtd.tex
\begin{tikzpicture}[scale=2] 

\node at (2, 0) {\textbf{Type $2200$}};
\begin{scope}[rotate around={45:(0,0)}]

\node (v0000) at (0,0) [circle, draw, fill=white, inner sep=3pt] {}; 
\node (v1000) at (1,0) [circle, draw, fill=white, inner sep=3pt] {}; 
\node (v0100) at (0,1) [circle, draw, fill=white, inner sep=3pt] {}; 
\node (v1100) at (1,1) [circle, draw, fill=white, inner sep=3pt] {}; 
\node (v2000) at (2,0) [circle, draw, fill=white, inner sep=3pt] {}; 
\node (v2100) at (2,1) [circle, draw, fill=white, inner sep=3pt] {}; 
\node (v0200) at (0,2) [circle, draw, fill=white, inner sep=3pt] {}; 
\node (v1200) at (1,2) [circle, draw, fill=white, inner sep=3pt] {}; 
\node (v2200) at (2,2) [circle, draw, fill=white, inner sep=3pt] {}; 

\node[below left] at (v0000) {$0000$};
\node[below] at (v1000) {$1000$};
\node[left] at (v0100) {$0100$};
\node[above left] at (v1100) {$1100$};
\node[below right] at (v2000) {$2000$};
\node[above right] at (v2100) {$2100$};
\node[above left] at (v0200) {$0200$};
\node[above] at (v1200) {$1200$};
\node[above right] at (v2200) {$2200$};

\draw (v0000) -- (v1000);
\draw (v0000) -- (v0100);
\draw (v1000) -- (v2000);
\draw (v1000) -- (v1100);
\draw (v0100) -- (v1100);
\draw (v0100) -- (v0200);
\draw (v1100) -- (v2100);
\draw (v1100) -- (v1200);
\draw (v2000) -- (v2100);
\draw (v0200) -- (v1200);
\draw (v1200) -- (v2200);
\draw (v2100) -- (v2200);

\end{scope}
\end{tikzpicture}

%% file: Itype31.tex
\begin{tikzpicture}[scale=2] 

\node at (1.5, 0) {\textbf{Type $3100$}};
\begin{scope}[rotate around={45:(0,0)}]

\node (v0000) at (0,0) [circle, draw, fill=white, inner sep=3pt] {}; 
\node (v1000) at (1,0) [circle, draw, fill=white, inner sep=3pt] {}; 
\node (v0100) at (0,1) [circle, draw, fill=white, inner sep=3pt] {}; 
\node (v1100) at (1,1) [circle, draw, fill=white, inner sep=3pt] {}; 
\node (v2000) at (2,0) [circle, draw, fill=white, inner sep=3pt] {}; 
\node (v2100) at (2,1) [circle, draw, fill=white, inner sep=3pt] {}; 
\node (v3000) at (3,0) [circle, draw, fill=white, inner sep=3pt] {}; 
\node (v3100) at (3,1) [circle, draw, fill=white, inner sep=3pt] {}; 

\node[below left] at (v0000) {$0000$};
\node[below] at (v1000) {$1000$};
\node[left] at (v0100) {$0100$};
\node[above left] at (v1100) {$1100$};
\node[below right] at (v2000) {$2000$};
\node[above right] at (v2100) {$2100$};
\node[below right] at (v3000) {$3000$};
\node[above right] at (v3100) {$3100$};

\draw (v0000) -- (v1000);
\draw (v0000) -- (v0100);
\draw (v1000) -- (v2000);
\draw (v1000) -- (v1100);
\draw (v0100) -- (v1100);
\draw (v1100) -- (v2100);
\draw (v2000) -- (v3000);
\draw (v2000) -- (v2100);
\draw (v2100) -- (v3100);
\draw (v3000) -- (v3100);

\end{scope}
\end{tikzpicture}

%% file: Itype211.tex
\begin{tikzpicture}[scale=2] 


\node at (1.5, 0) {\textbf{Type $2110$}};
\node (v0000) at (0,0) [circle, draw, fill=white, inner sep=3pt] {}; 
\node (v1000) at (-0.75,1) [circle, draw, fill=white, inner sep=3pt] {}; 
\node (v0100) at (0,1) [circle, draw, fill=white, inner sep=3pt] {}; 
\node (v0010) at (1,1) [circle, draw, fill=white, inner sep=3pt] {}; 
\node (v1100) at (-0.75,2) [circle, draw, fill=white, inner sep=3pt] {}; 
\node (v0110) at (1,2) [circle, draw, fill=white, inner sep=3pt] {}; 
\node (v1010) at (0.25,2) [circle, draw, fill=white, inner sep=3pt] {}; 
\node (v2000) at (-1.5,2) [circle, draw, fill=white, inner sep=3pt] {}; 
\node (v1110) at (0.25,3) [circle, draw, fill=white, inner sep=3pt] {}; 
\node (v2010) at (-0.5,3) [circle, draw, fill=white, inner sep=3pt] {}; 
\node (v2100) at (-1.5,3) [circle, draw, fill=white, inner sep=3pt] {}; 
\node (v2110) at (-0.5,4) [circle, draw, fill=white, inner sep=3pt] {}; 

\node[below] at (v0000) {$0000$};
\node[left] at (v1000) {$1000$};
\node[below] at (v0100) {$0100$};
\node[right] at (v0010) {$0010$};
\node[below] at (v1100) {$1100$};
\node[below right] at (v0110) {$0110$};
\node[below] at (v1010) {$1010$};
\node[left] at (v2000) {$2000$};
\node[right] at (v2100) {$2100$};
\node[below] at (v2010) {$2010$};
\node[right] at (v1110) {$1110$};
\node[above] at (v2110) {$2110$};

\draw (v0000) -- (v1000);
\draw (v0000) -- (v0100);
\draw (v0000) -- (v0010);
\draw (v1000) -- (v1100);
\draw (v1000) -- (v1010);
\draw (v1000) -- (v2000);
\draw (v0100) -- (v1100);
\draw (v0100) -- (v0110);
\draw (v0010) -- (v0110);
\draw (v0010) -- (v1010);
\draw (v2000) -- (v2100);
\draw (v1100) -- (v1110);
\draw (v0110) -- (v1110);
\draw (v1010) -- (v1110);
\draw (v1010) -- (v2010);
\draw (v2010) -- (v2110);
\draw (v1100) -- (v2100);
\draw (v2000) -- (v2010);
\draw (v2100) -- (v2110);
\draw (v1110) -- (v2110);

\end{tikzpicture}

%% file: Itype1111.tex
\begin{tikzpicture}[scale=2]

\node at (2, 0) {\textbf{Type $1111$}};

\node (v0000) at (0,0) [circle, draw, fill=white, inner sep=3pt] {}; 
\node (v1000) at (-1,1) [circle, draw, fill=white, inner sep=3pt] {}; 
\node (v0100) at (0,1) [circle, draw, fill=white, inner sep=3pt] {}; 
\node (v0010) at (1,1) [circle, draw, fill=white, inner sep=3pt] {}; 
\node (v1100) at (-1,2) [circle, draw, fill=white, inner sep=3pt] {}; 
\node (v0110) at (1,2) [circle, draw, fill=white, inner sep=3pt] {}; 
\node (v1010) at (0,2) [circle, draw, fill=white, inner sep=3pt] {}; 
\node (v1110) at (0,3) [circle, draw, fill=white, inner sep=3pt] {}; 

\node (v0001) at (2.5,1.0) [circle, draw, fill=white, inner sep=3pt] {}; 
\node (v1001) at (1.5,2.0) [circle, draw, fill=white, inner sep=3pt] {}; 
\node (v0101) at (2.5,2.0) [circle, draw, fill=white, inner sep=3pt] {}; 
\node (v0011) at (3.5,2.0) [circle, draw, fill=white, inner sep=3pt] {}; 
\node (v1101) at (1.5,3.0) [circle, draw, fill=white, inner sep=3pt] {}; 
\node (v0111) at (3.5,3.0) [circle, draw, fill=white, inner sep=3pt] {}; 
\node (v1011) at (2.5,3.0) [circle, draw, fill=white, inner sep=3pt] {}; 
\node (v1111) at (2.5,4.0) [circle, draw, fill=white, inner sep=3pt] {}; 

\draw (v0000) -- (v0001);
\draw (v1000) -- (v1001);
\draw (v0100) -- (v0101);
\draw (v0010) -- (v0011);
\draw (v1100) -- (v1101);
\draw (v0110) -- (v0111);
\draw (v1010) -- (v1011);
\draw (v1110) -- (v1111);

\draw (v0000) -- (v1000);
\draw (v0000) -- (v0100);
\draw (v0000) -- (v0010);
\draw (v1000) -- (v1010);
\draw (v1000) -- (v1100);
\draw (v0100) -- (v0110);
\draw (v0100) -- (v1100);
\draw (v0010) -- (v0110);
\draw (v0010) -- (v1010);
\draw (v1010) -- (v1110);
\draw (v1100) -- (v1110);
\draw (v0110) -- (v1110);

\draw (v0001) -- (v1001);
\draw (v0001) -- (v0101);
\draw (v0001) -- (v0011);
\draw (v1001) -- (v1011);
\draw (v1001) -- (v1101);
\draw (v0101) -- (v0111);
\draw (v0101) -- (v1101);
\draw (v0011) -- (v0111);
\draw (v0011) -- (v1011);
\draw (v1011) -- (v1111);
\draw (v1101) -- (v1111);
\draw (v0111) -- (v1111);

\node[below] at (v0000) {$0000$};
\node[left] at (v1000) {$1000$};
\node[right] at (v0100) {$0100$};
\node[right] at (v0010) {$0010$};
\node[left] at (v1100) {$1100$};
\node[left] at (v0110) {$0110$};
\node[below] at (v1010) {$1010$};
\node[above] at (v1110) {$1110$};
\node[below] at (v0001) {$0001$};
\node[right] at (v1001) {$1001$};
\node[right] at (v0101) {$0101$};
\node[right] at (v0011) {$0011$};
\node[left] at (v1101) {$1101$};
\node[right] at (v0111) {$0111$};
\node[below] at (v1011) {$1011$};
\node[above] at (v1111) {$1111$};

\end{tikzpicture}

%% file: Lemma_contradiction.tex
\begin{tikzpicture}[scale=2]

\node (vw) at (-0.6 - 1.5,2) [circle, draw, fill=white, inner sep=3pt] {};

\node (vy) at (-0.6 - 1.5,1) [circle, draw, fill=yellow, inner sep=3pt] {};

\node (vr) at (-0.6 - 1.5,1.5) [circle, draw, fill=red, inner sep=3pt] {};
\node at (-0.4 - 1.5, -1.75) {Fig. 2.\textbf{ The contradiction for Lemma 2.1, Case A. }};
\node (v310) at (-5 - 1.5,2) [circle, draw, fill=white, inner sep=3pt] {}; 
\node (v3-10) at (-1.7 - 1.5,2) [circle, draw, fill=white, inner sep=3pt] {}; 

\node (v300) at (-5 - 1.5,1) [circle, draw, fill=red, inner sep=3pt] {}; %

\node (v210) at (-3 - 1.5,1) [circle, draw, fill=yellow, inner sep=3pt] {}; %

\node (v2-10) at (-3.9 - 1.5,1) [circle, draw, fill=yellow, inner sep=3pt] {}; %

\node (v200) at (-3 - 1.5,0) [circle, draw, fill=white, inner sep=3pt] {}; 

\node (v100) at (-3 - 1.5,-1) [circle, draw, fill=white, inner sep=3pt] {}; 

%

\node[right] at (vw) {\mbox{ } vertices of C};

\node[right] at (vy) {\mbox{ } the remaining vertices of $C_1$};

\node[right] at (vr) {\mbox{ }  vertex of $C_1$  contradicting $c_1\leq 2$};

\node[below] at (v310) {$3100$};
\node[left] at (v3-10) {$3\text{-}100$};
\node[below] at (v300) {$3000$};
\node[right] at (v210) {$2100$};
\node[above] at (v2-10) {$2\text{-}100$};
\node[below] at (v200) {$2000$};
\node[below] at (v100) {$1000$};

\draw (v310) -- (v300);
\draw (v310) -- (v210);
\draw (v3-10) -- (v2-10);
\draw (v3-10) -- (v300);
\draw (v2-10) -- (v200);
\draw (v210) -- (v200);
\draw (v300) -- (v200);
\draw (v200) -- (v100);

\end{tikzpicture}

%% file: lemma_211contradiction.tex
\begin{tikzpicture}[scale=1.8]

\node (vw) at (-2.7 ,4) [circle, draw, fill=white, inner sep=3pt] {};

\node (vy) at (-2.7 ,3) [circle, draw, fill=yellow, inner sep=3pt] {};

\node (vr) at (-2.7 ,2) [circle, draw, fill=red, inner sep=3pt] {};

\node (vb) at (-2.7 ,1) [circle, draw, fill=blue, inner sep=3pt] {};

\node at (-3.5, -1) {Fig 3. \textbf{ The final contradiction in Lemma 4.2 for $c_{1}=1$.}};

\node[right] at (vw) {\mbox{ } vertices of C};

\node[right] at (vy) {\mbox{ }   vertices of $C_1$};

\node[right] at (vr) {\mbox{ }  vertices of $C_2$  contradicting $c_2=3$};

\node[right] at (vb) {\mbox{ } other vertices of C$_2$};

\node (v0000) at (-5.5,0) [circle, draw, fill=white, inner sep=3pt] {}; 
\node (v1000) at (-6.25,1) [circle, draw, fill=yellow, inner sep=3pt] {}; 
\node (v0100) at (-5.5,1) [circle, draw, fill=yellow, inner sep=3pt] {}; 
\node (v0010) at (-4.5,1) [circle, draw, fill=yellow, inner sep=3pt] {}; 
\node (v1100) at (-6.25,2) [circle, draw, fill=red, inner sep=3pt] {}; 
\node (v0110) at (-4.5,2) [circle, draw, fill=blue, inner sep=3pt] {}; 
\node (v1010) at (-5.25,2) [circle, draw, fill=red, inner sep=3pt] {}; 
\node (v2000) at (-7,2) [circle, draw, fill=blue, inner sep=3pt] {}; 
\node (v1110) at (-5.25,3) [circle, draw, fill=yellow, inner sep=3pt] {}; 
\node (v2010) at (-6,3) [circle, draw, fill=yellow, inner sep=3pt] {}; 
\node (v2100) at (-7,3) [circle, draw, fill=yellow, inner sep=3pt] {}; 
\node (v2110) at (-6,4) [circle, draw, fill=white, inner sep=3pt] {}; 

\node[below] at (v0000) {$0000$};
\node[left] at (v1000) {$1000$};
\node[below] at (v0100) {$0100$};
\node[right] at (v0010) {$\hspace{+1mm}0010$};
\node[below] at (v1100) {$1100$};
\node[below right] at (v0110) {$0110$};
\node[below] at (v1010) {$1010$};
\node[left] at (v2000) {$2000$};
\node[right] at (v2100) {$2100$};
\node[below] at (v2010) {$2010$};
\node[right] at (v1110) {$1110$};
\node[above] at (v2110) {$2110$};

\draw (v0000) -- (v1000);
\draw (v0000) -- (v0100);
\draw (v0000) -- (v0010);
\draw (v1000) -- (v1100);
\draw (v1000) -- (v1010);
\draw (v1000) -- (v2000);
\draw (v0100) -- (v1100);
\draw (v0100) -- (v0110);
\draw (v0010) -- (v0110);
\draw (v0010) -- (v1010);
\draw (v2000) -- (v2100);
\draw (v1100) -- (v1110);
\draw (v0110) -- (v1110);
\draw (v1010) -- (v1110);
\draw (v1010) -- (v2010);
\draw (v2010) -- (v2110);
\draw (v1100) -- (v2100);
\draw (v2000) -- (v2010);
\draw (v2100) -- (v2110);
\draw (v1110) -- (v2110);

\end{tikzpicture}

%% file: lemma_contr2112.tex
\begin{tikzpicture}[scale=2]

\node (vw) at (-2.7 ,4) [circle, draw, fill=white, inner sep=3pt] {};

\node (vy) at (-2.7 ,3) [circle, draw, fill=yellow, inner sep=3pt] {};

\node (vr) at (-2.7 ,2) [circle, draw, fill=red, inner sep=3pt] {};

\node (vb) at (-2.7 ,1) [circle, draw, fill=green, inner sep=3pt] {};

\node at (-3.5, -1) {Fig 4. \textbf{The final contradiction in Lemma 4.2 for $c_{2}=2$.}};

\node[right] at (vw) {\mbox{ } vertices of $C$};

\node[right] at (vy) {\mbox{ }   vertices of $C_1$};

\node[right] at (vr) {\mbox{ } vertex of $C_1$  contradicting $c_1=2$};

\node[right] at (vb) {\mbox{ } irrelevant vertices};

\node (v0000) at (-5.5,0) [circle, draw, fill=white, inner sep=3pt] {}; 
\node (v1000) at (-6.25,1) [circle, draw, fill=red, inner sep=3pt] {}; 
\node (v0100) at (-5.5,1) [circle, draw, fill=yellow, inner sep=3pt] {}; 
\node (v0010) at (-4.5,1) [circle, draw, fill=yellow, inner sep=3pt] {}; 
\node (v1100) at (-6.25,2) [circle, draw, fill=white, inner sep=3pt] {}; 
\node (v0110) at (-4.5,2) [circle, draw, fill=green, inner sep=3pt] {}; 
\node (v1010) at (-5.25,2) [circle, draw, fill=white, inner sep=3pt] {}; 
\node (v2000) at (-7,2) [circle, draw, fill=green, inner sep=3pt] {}; 
\node (v1110) at (-5.25,3) [circle, draw, fill=yellow, inner sep=3pt] {}; 
\node (v2010) at (-6,3) [circle, draw, fill=yellow, inner sep=3pt] {}; 
\node (v2100) at (-7,3) [circle, draw, fill=yellow, inner sep=3pt] {}; 
\node (v2110) at (-6,4) [circle, draw, fill=white, inner sep=3pt] {}; 

\node[below] at (v0000) {$0000$};
\node[left] at (v1000) {$1000$};
\node[below] at (v0100) {$0100$};
\node[right] at (v0010) {\hspace{+1mm}$0010$};
\node[below] at (v1100) {$1100$};
\node[below right] at (v0110) {$0110$};
\node[below] at (v1010) {$1010$};
\node[left] at (v2000) {$2000$};
\node[right] at (v2100) {$2100$};
\node[below] at (v2010) {$2010$};
\node[right] at (v1110) {$1110$};
\node[above] at (v2110) {$2110$};

\draw (v0000) -- (v1000);
\draw (v0000) -- (v0100);
\draw (v0000) -- (v0010);
\draw (v1000) -- (v1100);
\draw (v1000) -- (v1010);
\draw (v1000) -- (v2000);
\draw (v0100) -- (v1100);
\draw (v0100) -- (v0110);
\draw (v0010) -- (v0110);
\draw (v0010) -- (v1010);
\draw (v2000) -- (v2100);
\draw (v1100) -- (v1110);
\draw (v0110) -- (v1110);
\draw (v1010) -- (v1110);
\draw (v1010) -- (v2010);
\draw (v2010) -- (v2110);
\draw (v1100) -- (v2100);
\draw (v2000) -- (v2010);
\draw (v2100) -- (v2110);
\draw (v1110) -- (v2110);

\end{tikzpicture}

%% file: Dominos.tex
\begin{tikzpicture}[scale=2] 


\node at (6, -3) {Fig 5. \textbf{Domino property for all-ones CRC, Lemma 4.}  };
\begin{scope}[rotate around={0:(0,0)}]

\node at (10, 1) {\textbf{Vertices in $C$}  };

\node  at (8.7,1) [circle, draw, fill=white, inner sep=3pt] {};

\node at (10, 0) {\textbf{Vertices in $C_1$}  };

\node  at (8.7,0) [circle, draw, fill=yellow, inner sep=3pt] {};

\node at (6.5, -1) {\textbf{All orange vertices are   either in $C$ or in $C_2$ simultaneously}};

\node  at (2,-1) [circle, draw, fill=orange, inner sep=3pt] {};

\node at (6.5, -2) {\textbf{All black vertices are either in $C$ or in $C_2$ simultaneously}};

\node  at (2,-2) [circle, draw, fill=black, inner sep=3pt] {};



\node (v0002) at (5,2.5) [circle, draw, fill=yellow, inner sep=3pt] {}; 
\node (v0001) at (5,2) [circle, draw, fill=yellow, inner sep=3pt] {}; 
Distance 1
\node (v0000) at (5,1.5) [circle, draw, fill=white, inner sep=3pt] {}; 
\node (v000_1) at (5,1) [circle, draw, fill=white, inner sep=3pt] {}; 
\node (v000_2) at (5,0.5) [circle, draw, fill=yellow, inner sep=3pt] {};
\node (v000_3) at (5,0) [circle, draw, fill=yellow, inner sep=3pt] {};

\node (v0102) at (6.4,2.8) [circle, draw, fill=black, inner sep=3pt] {}; 
\node (v0101) at (6.4,2.3) [circle, draw, fill=black, inner sep=3pt] {}; 
Distance 1
\node (v0100) at (6.4,1.8) [circle, draw, fill=yellow, inner sep=3pt] {}; 
\node (v010_1) at (6.4,1.3) [circle, draw, fill=yellow, inner sep=3pt] {}; 
\node (v010_2) at (6.4,0.8) [circle, draw, fill=black, inner sep=3pt] {};
\node (v010_3) at (6.4,0.3) [circle, draw, fill=black, inner sep=3pt] {};

\node (v1002) at (3.5,2.5) [circle, draw, fill=orange, inner sep=3pt] {}; 
\node (v1001) at (3.5,2) [circle, draw, fill=orange, inner sep=3pt] {}; 
Distance 1
\node (v1000) at (3.5,1.5) [circle, draw, fill=yellow, inner sep=3pt] {}; 
\node (v100_1) at (3.5,1) [circle, draw, fill=yellow, inner sep=3pt] {}; 
\node (v100_2) at (3.5,0.5) [circle, draw, fill=orange, inner sep=3pt] {};
\node (v100_3) at (3.5,0) [circle, draw, fill=orange, inner sep=3pt] {};



\node[above left] at (v000_3) {$000{\text -}3$\,\,};
\node[above left] at (v000_2) {$000{\text - }2$\,\,};
\node[above left] at (v000_1) {$000{\text -}1$\,\,};
\node[above   left] at (v0000) {$0000$\,\,};
\node[above  left] at (v0001) {$0001$\,\,};
\node[above  left] at (v0002) {$0002$\,\,};

\node[above left] at (v100_3) {$100\text{-}3$\,\,};
\node[above left] at (v100_2) {$100\text{-}2$\,\,};
\node[above left] at (v100_1) {$100\text{-}1$\,\,};
\node[above   left] at (v1000) {$1000$\,\,};
\node[above  left] at (v1001) {$1001$\,\,};
\node[above  left] at (v1002) {$1002$\,\,};

\node[above right] at (v010_3) {$010{\text -}3$\,\,};
\node[above right] at (v010_2) {$010{\text -}2$\,\,};
\node[above right] at (v010_1) {$010{\text -}1$\,\,};
\node[above   right] at (v0100) {$0100$\,\,};
\node[above  right] at (v0101) {$0101$\,\,};
\node[above  right] at (v0102) {$0102$\,\,};

\node[above  right] at (7.5,2.5)  {$1$-slice\,\,};

\node[above  right] at (7.5,1.5)  {$0$-slice\,\,};

\node[above  right] at (7.5,0.5)  {${\text -}1$-slice\,\,};

\draw [dashed] (4.8,2.8)--(6.4,2.8);

\draw [dashed] (4.8,2.8)--(3.6,2.5);

\draw [dashed] (4.8,2.3)--(6.4,2.3);

\draw [dashed] (4.8,2.3)--(3.6,2.0);

\draw [dashed] (4.8,1.8)--(6.2,1.8);

\draw [dashed] (4.8,1.8)--(3.6,1.5);

\draw [dashed] (4.8,1.3)--(6.3,1.3);

\draw [dashed] (4.8,1.3)--(3.6,1);

\draw [dashed] (4.8,0.8)--(6.4,0.8);

\draw [dashed] (4.8,0.8)--(3.6,0.5);

\draw [dashed] (4.8,0.3)--(6.4,0.3);

\draw [dashed] (4.8,0.3)--(3.6,0);

\draw (v0000) -- (v000_1);
\draw (v0000) -- (v0100);
\draw (v0000) -- (v0001);

\draw (v0001) -- (v0002);
\draw (v0001) -- (v0101);
\draw (v0001) -- (v1001);

\draw (v1002) -- (v0002);
\draw (v0102) -- (v0002);

\draw (v000_1) -- (v100_1);
\draw (v000_1) -- (v010_1);

\draw (v0102) -- (v0002);

\draw (v0000) -- (v1000);

\draw (v000_2) -- (v100_2);
\draw (v000_2) -- (v010_2);
\draw (v000_2) -- (v000_3);

\draw (v000_2) -- (v000_1);

\draw (v000_3) -- (v100_3);
\draw (v000_3) -- (v010_3);
\draw [decorate,
    decoration = {brace}] (7.4,3) --  (7.4,2.3);
    \draw [decorate,
    decoration = {brace}] (7.4,2) --  (7.4,1.3);
    \draw [decorate,
    decoration = {brace}] (7.4,1) --  (7.4,0.3);
\end{scope}
\end{tikzpicture}

%% file: arxiv.bbl
\bigskip
\begin{thebibliography}{99}

\bibitem{Ax03} M. Axenovich, 
\href{https://doi.org/10.1016/S0012-365X(02)00744-6}{\it On multiple coverings of the infinite rectangular grid with balls of constant radius}, Discrete Math., {\bf 268}:1-3 (2003), 31--48.

\bibitem{ASV12} S.V. Avgustinovich, A.Y. Vasil'eva, I.V. Sergeeva, 
\href{https://doi.org/10.1134/S1990478912030027}{\it Distance regular colorings of an infinite rectangular grid}, J. Appl. Ind. Math., {\bf 6} (2012), 280--285.

\bibitem{AvgVas22} S.V. Avgustinovich, A.Yu. Vasil'eva, 
\href{https://doi.org/10.33048/semi.2022.19.072}{\it Completely regular codes in the n-dimensional rectangular grid}, Siberian Electronic Mathematical Reports, {\bf 19}(2) (2022), 861--869.

\bibitem{BKMTV21} E.A. Bespalov, D.S. Krotov, A.A. Matiushev, A.A. Taranenko, K.V. Vorob'ev, 
\href{https://doi.org/10.1002/jcd.21771}{\it Perfect 2-colorings of Hamming graphs}, J. Combin. Des., {\bf 29} (2021), 367--396.

\bibitem{BRZ19} J. Borges, J. Rifa, V.A. Zinoviev, 
\href{https://doi.org/10.1134/S0032946019010010}{\it On completely regular codes}, Probl. Inf. Transm., {\bf 55} (2019), 1--45.

\bibitem{DGHM09} P. Dorbec, S. Gravier, I. Honkala, M. Mollard, 
\href{https://doi.org/10.1007/s10623-009-9277-z}{\it Weighted codes in Lee metrics}, Des. Codes Cryptogr., {\bf 52} (2009), 209--218.

\bibitem{Et11} T. Etzion, 
\href{10.1109/TIT.2011.2161133}{\it Product Constructions for Perfect Lee Codes}, IEEE Trans. Inform. Theory, {\bf 57}:11 (2011), 7473--7481.

\bibitem{FDF07} D.G. Fon-Der-Flaass, 
\href{https://doi.org/10.1007/s11202-007-0075-4}{\it Perfect 2-colorings of a hypercube}, Sib. Math. J., {\bf 48} (2007), 740--745.

\bibitem{GW70} S.W. Golomb, L.R. Welch, 
\href{https://doi.org/10.1137/0118025}{\it Perfect Codes in the Lee Metric and the Packing of the Polyominoes}, SIAM J. Appl. Math., {\bf 18}:2 (1970), 302--317.

\bibitem{Kro9} D.S. Krotov, 
\href{https://arxiv.org/abs/0901.0004}{\it Perfect colorings of \( \mathbb{Z}^2 \): Nine colors}, arXiv:0901.0004 (Cornell Univ. Libr., Ithaca, NY, 2009).

\bibitem{KKM25} J.H. Koolen, D.S. Krotov, W.J. Martin, 
\href{https://www.researchgate.net/publication/382042905_Completely_Regular_Codes_and_Equitable_Partitions}{\it Completely regular codes: tables of small parameters for binary and ternary Hamming graphs}, in Completely Regular Codes in Distance-Regular Graphs, Chapman and Hall, 2025, 1--52.

\bibitem{Puz05} S. Puzynina, 
\href{https://www.elibrary.ru/item.asp?id=9534183}{\it Perfect colorings of graph \( G(\mathbb{Z}^2) \) into three colors}, in Russian, Discrete Analysis and Operation Research, {\bf 12}:2 (2005), 37--54.

\bibitem{Tarasov} A.A. Tarasov, 
{\it The parameters of distance-regular perfect coloring in the infinite tridimensional grid}, Master diploma, Novosibirsk State University.

\bibitem{Tar87} H. Tarnanen, 
\href{https://doi.org/10.1016/0166-218X(87)90063-1}{\it Upper bounds for constant weight and Lee codes slightly outside the Plotkin range}, Discrete Applied Mathematics, {\bf 16}:3 (1987), 265--277.

\bibitem{V22} A.Yu. Vasil'eva, 
\href{https://cyberleninka.ru/article/n/polnostyu-regulyarnye-kody-v-treugolnoy-reshetke/pdf}{\it Completely regular codes in the triangular grid}, Proceedings of MIPT, {\bf 14}:2 (2022), 14--42.

\end{thebibliography}
